\documentclass[12pt,a4paper]{amsproc}
\usepackage[english]{babel}
\usepackage[utf8]{inputenc}
\usepackage{latexsym}
\usepackage{amsmath}
\usepackage{amssymb}
\usepackage{amscd}
\usepackage{amsfonts}
\usepackage{dsfont}
\usepackage{titletoc}
\usepackage[usenames]{color}
\usepackage{amssymb,t1enc,amsthm,epsfig,amsmath,bm}
\usepackage{fancyhdr}
\usepackage{anysize}
\usepackage{multirow}
\usepackage{bbm,bbold}
\usepackage{blindtext}
\usepackage{parskip}
\usepackage{mathtools} 
\usepackage{array}
\usepackage{diagbox}
\usepackage{enumitem}
\usepackage{graphicx}
\usepackage[pdftex,plainpages=false,pdfpagelabels,colorlinks=true,citecolor=black,linkcolor=black,urlcolor=black,filecolor=black]{hyperref} 
\usepackage[mathscr]{eucal}
\theoremstyle{plain}
\newtheorem{quest}{Question}[section]
\newtheorem{question}[quest]{Question}

\newtheorem{defi}[quest]{Definition}

\newtheorem{teor}[quest]{Theorem}
\newtheorem{prop}[quest]{Proposition}

\newtheorem{coro}[quest]{Corollary}
\newtheorem{lemma}[quest]{Lemma}

\theoremstyle{remark}
\newtheorem{exam}[quest]{Example}
\newtheorem{obs}[quest]{Remark}

\newtheorem*{claim}{Claim}
\newtheorem*{teor1}{Theorem}

\newcommand{\re}{\ensuremath{\mathbbm{R}}}
\newcommand{\ene}{\ensuremath{\mathbbm{N}}}
\newcommand{\partes}{\ensuremath{\mathbbm{P}}}
\newcommand{\nnormbuno}{\ensuremath{\mathbbm{D}}} 
\newcommand{\nnormb}{\ensuremath{ \mathcal{D} }} 

\newcommand{\nnomega}[1]{\ensuremath{(\mathcal{D}_{#1})^\omega}}
\newcommand{\blocksubspaces}{\ensuremath{ \mathcal{F} }}
\newcommand{\nnormbasis}{\ensuremath{ \mathcal{B} }} 
\newcommand{\admfam}{\ensuremath{ \mathfrak{B} }} 

\newcommand{\eneinf}{\ensuremath{[\mathbbm{N}}]^{\infty}}
\newcommand{\enefin}{\ensuremath{[\mathbbm{N}}]^{<\infty}}

\newcommand{\nn}{\ensuremath{\mathcal{N}}}

\newcommand{\baseen}{\ensuremath{(e_n)_n}}
\newcommand{\basexn}{\ensuremath{(x_n)_n}}
\newcommand{\baseyn}{\ensuremath{(y_n)_n}}
\newcommand{\basezn}{\ensuremath{(z_n)_n}}
\newcommand{\baseun}{\ensuremath{(u_n)_n}}
\newcommand{\basevn}{\ensuremath{(v_n)_n}}
\newcommand{\basewn}{\ensuremath{(w_n)_n}}
\newcommand{\ui}{\ensuremath{(u_i)_i}}
\newcommand{\vi}{\ensuremath{(v_i)_i}}
\newcommand{\wi}{\ensuremath{(w_i)_i}}

\newcommand{\Chi}{\ensuremath{\mbox{\large $\chi$}}}

\newcommand{\summ}[4]{\ensuremath{ \sum_{{#1} = {#2}}^{{#3}} {#4} }}
\newcommand{\normm}[1]{\ensuremath{\|{#1} \|}}
\newcommand{\normmm}[1]{{\left\vert\kern-0.25ex\left\vert\kern-0.25ex\left\vert #1 \right\vert\kern-0.25ex\right\vert\kern-0.25ex\right\vert}}

\newcommand{\tends}[3]{\ensuremath{\xrightarrow[{#1} \rightarrow {#2}]{} {#3}}}
\newcommand{\supp}[2]{\ensuremath{supp_{{}_{#1}}({#2}) }}

\newcommand{\sph}{\ensuremath{\mathbbm{S}}}
\newcommand{\ba}{\ensuremath{\mathbbm{B}}}
\newcommand{\aset}{\ensuremath{\mathcal{A}}}
\newcommand{\asetf}[1]{\ensuremath{\mathcal{A}_{#1}}}
\newcommand{\asetfin}[1]{\ensuremath{[\mathcal{A}_{#1}]}}

\newcommand{\hast}{\ensuremath{H_{Y,X}^\aset}}
\newcommand{\hastml}{\ensuremath{H_{L,M}^\aset}}

\newcommand{\embast}{\ensuremath{ \overset{ \tiny{\aset}}{\hookrightarrow}}}
\newcommand{\embs}{\ensuremath{ \overset{\text{\tiny{s}}}{\hookrightarrow}}}

\newcommand{\nembast}{\ensuremath{ \not \overset{\tiny{\aset}}{\hookrightarrow}}}
\newcommand{\gast}{\ensuremath{G_{Y,X}^\aset}}
\newcommand{\gastml}{\ensuremath{G_{L,M}^\aset}}

\newcommand{\seqq}[1]{\ensuremath{\mathfrak{#1}}}

\bigskip

\bigskip

\title{Tight - minimal dichotomies in Banach spaces}

\author{Alejandra C. C\'aceres-Rigo, Valentin Ferenczi}

\address{Departamento de Matem\'atica, Instituto de Matem\'atica e
Estat\'\i stica, Universidade de S\~ao Paulo, rua do Mat\~ao 1010,
05508-090 S\~ao Paulo SP, Brazil.}
\email{caceres.rigo@gmail.com, ferenczi@ime.usp.br }



\thanks{The first author was supported by the São Paulo Research Fundation (FAPESP), project 2017/18976-5. The second author was supported by the São Paulo Research Fundation (FAPESP), projects 2016/25574-8 and 2022/04745-0, and by the National Council for Scientific and Technological Development (CNPq), grant 303731/2019-2.}

\begin{document}

\begin{abstract} 
 We extend the methods used by V. Ferenczi and Ch. Rosendal to obtain the `third dichotomy' in the program of classification of Banach spaces up to subspaces, in order to prove that a Banach space $E$ with an admissible system of blocks $(\nnormb_E, \asetf{E})$, contains an infinite dimensional subspace with a basis which is either $\asetf{E}$-tight or $\asetf{E}$-minimal. In this setting we obtain, in particular, dichotomies regarding subsequences of a basis, and as a corollary, we show that every normalized basic sequence $\baseen$ has a subsequence which satisfies a tigthness property or is spreading. Other dichotomies between notions of minimality and tightness are demonstrated, and the Ferenczi - Godefroy interpretation of tightness in terms of Baire category is extended to this new context.
\end{abstract}

\maketitle

\tableofcontents

\section{Introduction}
In this paper, when we refer to a Banach space, we are considering  a separable infinite dimensional Banach space.  Subspaces of Banach spaces are assumed to be infinite dimensional and closed, unless stated otherwise. 
In \cite{Gow02} W. T. Gowers began the {\it Classification Program of Banach spaces up to subspaces}. The program aims to classify Banach spaces into ``inevitable'' classes, using dichotomies between two opposite inevitable classes of Banach spaces. Conditions for a class to be considered of interest for the Program were given by Gowers: the classes must be {\bf inevitable}, that is, every Banach space must belong to a class. A class must be {\bf hereditary} for closed subspaces or, if the property that determines the class is defined for basic sequences, then the class must be hereditary for block subspaces. Two different classes must be {\bf disjoint}. The property that determines the class {\bf must give additional information} about the space of operators defined over the space or over its subspaces. 

A Banach space $X$ is decomposable if it can be written as the direct sum of two closed infinite-dimensional subspaces, otherwise $X$ is said indecomposable. A Banach space is said hereditarily indecomposable (or HI) if all its infinite-dimensional subspaces are indecomposable. Gowers showed a {\it first dichotomy} (see \cite{Gow96}) giving the first two examples of inevitable classes: every Banach space has a separable subspace that is either hereditary indecomposable, or has an unconditional basis. In \cite{Gow02} were proved a {\it second dichotomy}: Any Banach space contains a subspace with a basis such that no pair of disjointly supported block subspaces are isomorphic, or any two block subspaces have isomorphic subspaces. In resume, in \cite{Gow02}, there were presented four inevitable classes.

Later in \cite{FerencziRosendal1}, V. Ferenczi and Ch. Rosendal proved three new dichotomies. They refined in that work the list of inevitable classes into six main classes and 19 secondary classes. The main result in \cite{FerencziRosendal1} is a {\it third dichotomy}, which contrasts the dual notions of minimality and tightness and is central for this work:
\begin{teor}[Third dichotomy, \cite{FerencziRosendal1}]
    If $E$ is a Banach space, then $E$ contains a subspace with basis which is either tight or minimal.
\end{teor}
It is well known that a Banach space is minimal if it can be isomorphically embedded in any of its  subspaces. Let us suppose that $E$ is a Banach space with Schauder basis $\baseen$.  A  Banach space $Y$ is tight in $E$  (see \cite{FerencziRosendal1} for this definition and an extensive study of this notion) if there is a sequence $(I_n)_n$ of successive finite subsets of $\ene$,  such that for every $A$ infinite subset of $\ene$, $Y$ can not be isomorphically embedded in $[e_n, n \notin \cup_{i \in A}I_i]$. A basis $(e_n)_n$ is tight for $E$ if any Banach space $Y$ is tight in $E$, and $E$ is tight if it has a tight basis.

A useful characterization of tightness was given  in \cite{TightBaire} using Baire category: $Y$ is tight in $E = [e_n]_n$ if, and only if, the set of indices $A \subseteq \ene$  for which $Y$ can be embedded in $ [e_n: n \in A]$ is meager in $\partes(\ene)$ (after the natural identification of $\partes(\ene)$ with the Cantor space $2^\omega$, via characteristic functions). 

Tightness is an opposite notion to minimality: it is clear that a tight space can not be minimal, nor a minimal space can have a tight subspace. In both definitions, tight and minimal spaces, the underlying embedding is the isomorphic embedding. We say that $Y=[y_n]_n$ isomorphically embeds in $E=[e_n]_n$ if $\baseyn$ is equivalent to a (basic) sequence $\basexn$ in $E$. After a standard perturbation argument, one can ask that such basic basic sequence is a sequence of finitely supported vectors of $E$. One can consider different forms of embedding of $Y$ into $E$, depending of the  properties the basic sequence $\basexn$ in $E$ satisfies. For example, one can ask $\basexn$ to be a a block sequence of the basis $\baseen$ of $E$ or a sequence of disjointly supported vectors in $E$.

The authors in \cite{FerencziRosendal1} also stated that after a variation of the notion of embedding in the definition of tight basis, and consequently, modifying the methods involved in the proof of the {\it third dichotomy}, the following result is obtained:

\begin{teor1}[Theorem 3.16, \cite{FerencziRosendal1}]\label{teordichogrande}
    Every Banach space with a basis contains a block subspace $E=[e_n]_n$ satisfying one of the following properties:
    \begin{itemize}
        \item[(1)] For any  $[y_n]_n\leq E$, there is a sequence $(I_n)_n$ of successive intervals in $\ene$ such that for any $A \in \eneinf$, $[y_n]_n$ does not embed into $[e_n, n \notin \cup_{i \in A}I_i]$, as a sequence of disjointly supported vectors, respectively as a block sequence.
        \item[(2)] For any $[y_n]_n \leq E$, $(e_n)_n$ is equivalent to a sequence of disjointly supported vectors of $[y_n]_n$, respectively $(e_n)_n$ is equivalent to a block sequence of $[y_n]_n$.
    \end{itemize}
\end{teor1}


Therefore, modifying the embedding we obtain a corresponding type of minimality and its associated dual type of tightness. In this work we define and study different types of minimality and their respective dual notions of tightness, in order to obtain new dichotomies between them. An additional attractive aspect of  this point of view is to allow us to extend the techniques to the study of subsequences of a given basis, instead of subspaces of a given space.

Those ways of interpreting the embedding are coded in what we call an admissible system of blocks, which is a pair $(\nnormb_E, \asetf{E})$ associated to a Banach space $E$ with a fixed normalized basis $\baseen$. Basically, a set of blocks (see Definition \ref{setofblocks}) $\nnormb_E$ for $E$ is a set ``containing'' the possible bases of the block subspaces one admits to consider. Meanwhile,  an admissible set (see Definition \ref{admpro}) $\asetf{E}$ for $E$ is the set of infinite sequences of vectors which are the images of the embedding one wants to consider. Using this coding in the case of ``being equivalent to a subsequence of $\baseen$'', for example, $\nnormb_E$ would be the set which elements are the vectors of the basis and $\asetf{E}$ is the set of all the subsequences of $E$. The properties of set of blocks and admissible families will be studied in Section  \ref{sec:Admissiblesystem}.

This coding for embedding through admissible sets $\asetf{E}$ of vectors naturally leads us to define the notions of $\asetf{E}$-minimality and $\aset$-tightness, which depend on the pair $(\nnormb_E, \asetf{E})$, as follows: given a set of blocks $\nnormb_E$ and an admissible set $\asetf{E}$ for $E$, we say that $E$ is $\asetf{E}$-minimal if for every block sequence $\basexn \in (\nnormb_E)^\omega $, there is a sequence $\baseyn \in \asetf{E} \cap X^\omega$ equivalent to $\baseen$. We say that $\baseen$ is a $\asetf{E}$-tight basis for $E$ if for every Banach space $Y$ there is a sequence of successive intervals $(I_i)_i$ such that for every $A$ infinite subset of $\ene$:
    \begin{equation}
        Y  \nembast [e_n: n \notin \cup_{i \in A} I_i].
    \end{equation}
    
    The study of $\asetf{E}$-embeddings and $\asetf{E}$-minimality is done in Sections \ref{sec:embedding}, \ref{sec:blocksubspaceinF} and summarized in Section \ref{sec:typesofminimality}. Basic properties of $\asetf{E}$-tight bases are studied in Section \ref{sec:Atightness}.

    In this work, we generalize the methods in \cite{FerencziRosendal1} to use this admissible systems and we prove the main theorem of this work:

\begin{teor}\label{teor3.13}
    Let $E$ be a Banach space with normalized basis $\baseen$ and  $(\nnormb_E, \asetf{E})$ be an admissible system of blocks for $E$. Then $E$ contains a $\nnormb_E$-block subspace which is either $\asetf{E}$-tight  or $\asetf{E}$-minimal.
\end{teor}

The authors of \cite{FerencziRosendal1} stated that modifying the notion of embedding in the definition of tight basis, and consequently modifying the methods involved in the proof of the Third Dichotomy, the following statement is true:

\begin{center}
    \it{ Every Banach space with a basis contains a block subspace $E=[e_n]_n$ satisfying that either for any  $[y_n]_n\leq E$, there is a sequence $(I_n)_n$ of successive intervals in $\ene$ such that for any $A \in \eneinf$, $[y_n]_n$ does not embed into $[e_n, n \notin \cup_{i \in A}I_i]$ as a permutation of a block sequence; or for any $[y_n]_n \leq E$, $(e_n)_n$ is permutatively equivalent to a block sequence of $[y_n]_n$.}
\end{center}

But, as we saw in Proposition \ref{remperm1} a basic sequence $\baseyn$ being embedded in $[e_n]_n=E$ as a permutation of $\baseen$,  is not an $\asetf{E}$-embbedding obtained from an admissible set for $E$, and this is fundamental for the proofs in this chapter to work.  We have no evidence that in this case the above affirmation is true, but it can not be obtained just by modifying the embedding notion in the proof of  the third dichotomy, as the authors of \cite{FerencziRosendal1} claim.

\subsection{A sketch of the proof}

 The main tool used in \cite{FerencziRosendal1} in order to prove the third dichotomy is the notion of {\it generalized asymptotic game} which is a generalization of the notion of infinite asymptotic game (see  \cite{Rosendal2006InfiniteAG} and \cite{OdellSchlumprecht}). A modification of the infinite asymptotic game was first defined by Ferenczi in \cite{Fer} to prove that a space saturated with subspaces with a Schauder basis, which embed into the closed linear span of any subsequence of their basis,  must contain a minimal subspace. The work in \cite{Fer} generalized the methods and the result of Pelczar in \cite{pelczar}: a Banach space saturated with subsymmetric basic sequences contains a minimal subspace. 

Let $X=[x_n]_n$ and $Y=[y_n]_n$ be two block subspaces of a Banach space $E$ with Schauder basis $\baseen$. The {\it generalized asymptotic game} $H_{Y,X}$ with constant $C$  is a game with infinite rounds between player $I$ and player $II$ where in the $k$-th round, player $I$ moves a natural number $n_k$ and player $II$ responds with a natural number $m_k$ and a not necessarily normalized finitely supported vector $u_k$ such that $\supp{}{u_k} \subseteq \cup_{i=0}^k [n_i, m_i]$. The outcome of the game is the not necessarily block sequence $\baseun$. Player $II$ wins the game if $\baseyn$ is $C$-equivalent to $\baseun$.

In order to prove Theorem \ref{teor3.13}, we follow the demonstration of the third dichotomy  generalizing the arguments for the context of $\asetf{E}$-minimality and $\asetf{E}$-tightness, creating the notion of ``admissible systems of blocks''. In the first place, we shall adapt for $\nnormb_E$-block subspaces two technical Lemmas (Lemma \ref{lemma2.2} and Lemma \ref{lema3.14}), whose original versions for block subspaces were proved in \cite{FerencziRosendal1} and in \cite{procpelczar}, respectively. 
We define an ${\mathcal A}$-version of the generalized asymptotic game $\hast$ with constant $C$,  depending on an admissible set $\asetf{E}$, requiring that the outcome $\baseun$ of the game to be an element of $\asetf{E}\cap X^\omega$. Again, the game $\hast$ with constant $C$ is open for player $I$ and then, by the determinacy of open Gale-Stewart games, is determined. 

In Section \ref{sec:relagametight} we prove technical lemmas varying the methods of Ferenczi and Rosendal: we show that if $E$ is in some way saturated by $\nnormb_E$-block subspaces $X$ and $Y$ such that player $I$ has a winning strategy for the game $\hast$ with constant $C$, then $E$ has an $\asetf{E}$-tight subspace. 

Before the proof of our main theorem it is necessary to introduce two games for $\asetf{E}$-minimality: the game $\gast$ with constant $C$ and a version of that game assuming that finitely many moves have been made in $\gast$. This shall be done in Section \ref{sec:gamesforminimality}. The main result in that section relates the existence of a winning strategy for player ${\rm II}$ in the game $\hast$ with the existence of a winning strategy for player ${\rm II}$ in the game $\gast$. Finally, after the proof of Theorem \ref{teor3.13} which will be showed in Section \ref{proofofteor}, some tight - minimal dichotomies are presented.

\section{Preliminaries}
\label{sec:preliminares}

If $E$ is a Banach space $\sph_E$, $\ba_E$ and $\overline{\ba}_E$ denote the unit sphere, the open and closed ball of $E$, respectively. For $\varepsilon > 0$ and $x \in E$, $\ba_E(x,\varepsilon)$ and $\overline{\ba}_E(x,\varepsilon)$ denote the open and closed ball in $E$ centered in $x$ with radius $\varepsilon$, respectively.

Suppose that $\baseen$ is a basis for the Banach space $E$. We define the support of $x \in E$ (in symbols $\supp{E}{x}$) in the basis $\baseen$ as the set $\{n \in \ene: e_{n}^\ast(x) \neq 0\}$, where $e_{k}^{\ast}$ are the coordinate functionals defined by $x = \summ{n}{0}{\infty}{\lambda_n e_n} \mapsto \lambda_k $,  for all $k \in \ene$. The support of the zero vector of $E$ is the empty set.

We say that a Banach space $X$ is isomorphic to a Banach space $Y$ with constant $K$ (denoted as $Y \simeq_K X$) if there exists a one-to-one bounded linear operator $T$ from $X$ onto $Y$ such that $T^{-1}$ is bounded and $K \geq \normm{T}\cdot \normm{T^{-1}}$.  We say that $X$ contains a $K$-isomorphic copy of $Y$, or $Y$ is $K$-embeddable in $X$ (denoted as $Y \hookrightarrow_K X$), if $Y \simeq_K Z$ for some $Z$ subspace of $X$. Finally,  $Y$ is isomorphically embeddable in $X$, or just embeddable, (in symbols $Y \hookrightarrow X$), if $Y \hookrightarrow_K X$, for some $K \geq 1$. In this case we say that $X$ contains an isomorphic copy, or just a copy, of $Y$. 

Let  $K\geq 1$, two basic sequences $\basexn$ and $\baseyn$ are $K$- equivalent ($\basexn \sim_K \baseyn$) if for all $k \in \ene$ and $(a_i)_{i=0}^k$ finite sequence of scalars we have
    $$\frac{1}{K}\normm{\summ{n}{0}{k}{a_n x_n}} \leq \normm{\summ{n}{0}{k}{a_n y_n}} \leq K \normm{\summ{n}{0}{k}{a_n x_n}}.$$
Two basic sequences are equivalent if there is $K\geq 1$ satisfying that such sequences are $K$-equivalent.

We shall use the following well known result:
\begin{prop}\label{cdem}
    Let $X$ be a Banach space with basis $\basexn$ with basis constant $C$ and let $M\geq 1$. Then, there is a constant $c \geq 1$ which depends on $C$ and $M$, such that if $(z_n)_n $ and $\baseyn$ are normalized block bases of $\basexn$ which differ only in $M$ terms, then $\baseyn \sim_{c} (z_n)_n$.  
\end{prop}

Let $A$ be a non-empty set. $|A|$ denotes the cardinality of $A$, $\partes(A)$ denotes the power set of $A$, $[A]^{<\infty}$ and $[A]^{\infty}$ denote the set of finite subsets of $A$ and the set of infinite subsets of $A$, respectively.
Given $A, B \subset\ene$ and assuming that $\max A$ and $\min B$ exist, we write $A<B$ to mean that $\max A < \min B$. So, when we refer to a sequence $(I_n)_n$ of successive finite subsets of $\ene$, we are saying that $I_n < I_{n+1}$ for every $n\in \ene$. Also, when we refer to an interval $I$  of natural numbers, we are meaning that $I = [a, b]\cap \ene$, for some $0\leq a<b$.
Let us denote the set of nonempty finite sets of $\ene$ by ${\rm FIN}$, that is  ${\rm FIN} := \enefin \setminus \{\emptyset\}$. We denote by ${\rm FIN}^\omega$ the set of infinite sequences of non-empty finite subsets of $\ene$.

We shall consider the Cantor space $2^\omega = \{0,1\}^\omega$ with its natural product topology after endow in $\{0,1\}$ the discrete topology. If $\seqq{s}=(s_i)_i \in 2^\omega$, define $\supp{}{\seqq{s}} = \{i\in \ene: s_i  = 1\}$. Notice that $\partes(\ene)$ can be identified with $2^\omega$ using the characteristic functions: let $A \in \partes(\ene)$, then the characteristic function $\Chi_A$ belongs to $2^\omega$ and $A = \supp{}{\Chi_A}$. Thus, families of subsets of $\ene$ sometimes will be seen as families of sequences of $\ene$. Therefore, any $\mathcal{F} \subseteq \partes(\ene)$ can be seen as a topological subspace of $2^\omega$. A basic open subset of $2^\omega$ determined by $\seqq{s} \in 2^\omega $ and $J \in \enefin$ is given by
$$\nn_{\seqq{s},J} := \{\seqq{u} = \baseun \in 2^\omega :  \forall n \in J (u_n = s_n) \}.$$

\subsection{A law and Ramsey-like theorems}

A Polish space is a separable completely metrizable topological space. In this subsection we shall recall some classical theorems. The next theorem is known as the first topological 0-1 law:

\begin{teor}[(8.46) in \cite{Kechris}]\label{01topologicallaw}
    Let $X$ be a Polish space, and $G$ be a group of homeomorphisms of $X$ with the following property: for any $U$ and $V$ non-empty open subsets of $X$, there is $g \in G$ such that $g(U) \cap V \neq \emptyset$. If $A\subseteq X$ has the Baire Property and is $G$-invariant (i.e. g(A) = A, for every $g \in G$), then $A$ is meager or comeager in $X$.
\end{teor}

The next theorem is known as the Galvin-Prikry's Theorem:

\begin{teor}[(19.11) in \cite{Kechris}]\label{galvinprikry}
    Let $\eneinf = P_0 \cup ... \cup P_{k-1}$, where each $P_i$ is Borel and $k \in \ene$. Then there is $H \in \eneinf$ and $i < k$ with $[H]^{\infty} \subseteq P_i$.
\end{teor}

According to Ramsey Theory's nomenclature, a subset $\mathcal{C}$ of $\eneinf$ is Ramsey if, and only if, there is some $H \in \eneinf$ such that $[H]^\infty \subseteq \mathcal{C}$ or  $[H]^\infty \subseteq \eneinf \setminus \mathcal{C}$. So, Galvin-Prikry's Theorem can be enunciated as follows: the Borel sets of $\eneinf$ are Ramsey. The next theorem, Silver's Theorem, says that analytic subsets of $\eneinf$ are completely Ramsey, which implies that analytic subsets of $\eneinf$ are Ramsey (since all completely Ramsey subsets are Ramsey). We recommend the lecture of \cite{Kechris} for more information about this definitions and proofs.

\begin{teor}[(29.8) in \cite{Kechris}]\label{silver}
   Analytic subsets of $\eneinf$ are  completely Ramsey.
\end{teor}

\section{Admissible sets and families}
\label{sec:Admissiblesystem}

Along this section suppose $E$ is a Banach space with Schauder basis $\baseen$. Set $\nnormbasis_E :=\{e_n:n \in \ene\}$ and  $\nnormbasis_{E}^{\pm} := \{e_n:n \in \ene\}\cup \{- e_n:n \in \ene\} $. Let $\mathbf{F}_E$ be a countable subfield of $\re$ containing the rationals such that for all $\summ{i}{0}{n}{\lambda_i e_i}$, with $n \in \ene$ and $(\lambda_i)_{i=0}^n \in (\mathbf{F}_{E})^{n+1}$, the norm $\normm{\summ{i}{0}{n}{\lambda_i e_i}} \in \mathbf{F}_E$. We denote by $\nnormbuno_E$ the countable set of nonzero not necessarily normalized finite $\mathbf{F}_E$-linear combinations of $\baseen$.

\subsection{Definitions and notations}\label{subsec:defssystemblocks}
\begin{defi}\label{defast}
    Let $\basexn$ be a sequence of successive finitely supported vectors of $E$. For $X = [x_n]_n$, let us define the operation $\ast_X: (\nnormbuno_E \cap X)^\omega \times(\nnormbuno_E)^\omega \rightarrow (\nnormbuno_E)^\omega$ as follows: if $v = \basevn $ belong to $ (\nnormbuno_E)^\omega $ and $u = \baseun \in (\nnormbuno_E \cap X)^\omega $ such that  for each $n \in \ene$ 
    \begin{equation*}
        u_n = \sum_{i \in \supp{X}{u_n}} \lambda_{i}^n x_i,
    \end{equation*}
    then $u \ast_X v $ is the sequence $\basewn$, such that for each $n \in \ene$
    $$w_n = \sum_{i \in \supp{X}{u_n}} \lambda_{i}^n v_i.$$
\end{defi}

Notice that the set $\nnormbuno_E \cap X$ could be empty. In our work we shall take subspaces generated by vectors on $\nnormbuno_E$, so this will not occur.

\begin{defi}\label{setofblocks}
    We define a set of blocks for the space $E$ to be a set $\nnormb_E$ satisfying the following conditions 
    \begin{enumerate}
        \item[a)] $\nnormb_E$ is a subset of $\nnormbuno_E$.
        \item[b)]  The set $\{e_n: n \in \ene\} $ is contained in $ \nnormb_E$.
        \item[c)] If $u \in \nnormb_E$, then $\frac{u}{\|u\|} \in \nnormb_E$. 
        \item[d)] For every $\baseun \in (\nnormb_E)^\omega$ and $\basevn \in (\nnormb_{E})^\omega$, we have $ \baseun \ast_E \basevn \in (\nnormb_{E})^\omega.$
        \item[e)] Let $(x_i)_{i=0}^{n} \in (\nnormb_E)^{n+1} $ with $x_i < x_{i+1}$ for every $0\leq i \leq n$, and $X=[x_i]_{i\leq n}$. If $ u \in  \nnormb_E $ is such that 
        $$u = \sum_{i =0}^{n} \lambda_{i} x_i,$$
        then 
         $$v = \sum_{i =0}^{n} \lambda_{i} e_i \in \nnormb_E.$$
    \end{enumerate}
    We say that a vector $u$ is a  $\nnormb_E$-block if  $u$ is an element of the set $\nnormb_E$.
\end{defi}

\begin{exam}
    $ \nnormbasis_E$, $\nnormbasis_{E}^{\pm} $ and $\nnormbuno_E$ are sets of blocks for $E$. 
\end{exam}

\begin{defi}\label{blocksequence}
    Let $D \subseteq \nnormbuno_E$ be an infinite subset such that $D^\omega$ contains a block basis of $\baseen$.
    \begin{itemize}
        \item[$(i)$] We say that $\baseyn \in  E^\omega$ is a $D$-block sequence if, and only if, $\baseyn$ is a block basis of $\baseen$ and for each $n \in \ene$ we have $y_n \in D$.
        \item[$(ii)$] A subspace $Y$ is a $D$-block subspace if it is the closed subspace spanned by a $D$-block sequence. $\baseyn$. 
    \end{itemize}
     Without loss of generality we shall suppose that a $D$-block subspace is always generated by a normalized $D$-block sequence. We say that $Y=[y_n]_n$ is a $D$-block subspace when $Y$ is a $D$-block subspace generated by the normalized $D$-block sequence $\baseyn$.
\end{defi}

\begin{defi}
    Let $\nnormb_E$ be a set of blocks for $E$. Let $X$ be a $\nnormb_E$-block subspace. 
    \begin{itemize}
        \item[$(i)$] We define $\nnormbuno_X := \nnormbuno_E \cap X$.
        \item[$(ii)$]We denote by $\nnormb_X := \nnormb_E \cap X$  the set formed by the blocks which belong to $X$.
        \item[$(iii)$] We denote by $bb_{\nnormb}(E)$ the set of normalized $\nnormb_E$-block sequences of  $E$, i.e.
        $$ bb_\nnormb(E) := \{\basexn \in (\nnormb_E)^\omega : \basexn \text{ is a } \nnormb_E\text{-block sequence of } E \;\& \; \forall n \in \ene (\normm{x_n} = 1) \}.$$
        \item[$(iv)$] 
        We denote by $bb_{\nnormb}(X)$ the set of normalized
        $\nnormb_X$-block sequences of $E$, i.e.
        $$ bb_\nnormb(X) := \{\baseyn \in (\nnormb_X)^\omega : \baseyn \text{ is a } \nnormb_X\text{-block sequence of } E \;\& \; \forall n \in \ene (\normm{x_n} = 1) \}.$$
    \end{itemize}
\end{defi}

If $\nnormb_E$ is a set of blocks for $E$ and $X$ is a $\nnormb_E$-block subspace, then we sometimes identify an element $\baseyn$ of $bb_\nnormb(X)$ with the $\nnormb_E$-block subspace that it generates. 

We endow $(\nnormb_E)^\omega$  with the product topology obtained by considering $\nnormb_E$ with the discrete topology.  $(\nnormb_E)^\omega$ is a Polish space. Also, the set $(\ene \times \ene \times \nnormb_E)^\omega$ with its natural product topology is Polish. The set $bb_{\nnormb}(E)$ is a non-empty closed topological subspace of $\nnomega{E}$, therefore, it is Polish.

\begin{defi}\label{admpro}
    Let $\nnormb_E$ be a set of blocks for $E$. We say that a set $\asetf{E}$ is an admissible set for $E$ if, and only if, it satisfies the following conditions:
    \begin{itemize}
        \item[a)]\label{defprop} $\asetf{E}$ is a closed subset of $(\nnormb_E)^\omega$.
        \item[b)] $\asetf{E}$ is a subset of $(\nnormb_E)^\omega$ such that  it contains all the $\nnormb_E$-block sequences.
        \item[c)] For every $\baseyn \in \asetf{E}$ and every $\nnormb_E$-block subspace  $X=[x_n]_n$ we have that if $\baseun \in {(\nnormb_X)}^\omega$, then
        $$ \baseun \in \asetf{E} \iff \baseun \ast_X \baseyn  \in \asetf{E}. $$
        \item[d)] Let $\baseyn$ be a $\nnormb_E$-block sequence and $Y = [y_n]_n$. For every $\baseun \in \asetf{E}$ and $k \in \ene$, there is $\basevn \in Y^\omega$ such that $(u_0, ..., u_k, v_0, v_1, ...) \in \asetf{E}$.
    \end{itemize}
\end{defi}

\begin{defi}
 Let $\nnormb_E$ be a set of blocks for $E$, $\asetf{E}$ an admissible set for $E$ and $X$ be a $\nnormb_E$-block subspace.
    \begin{itemize} 
        \item[$(i)$]  We denote  by $\asetf{X}$ the intersection $\asetf{E} \cap X^\omega$.
        \item[$(ii)$] We denote by $\asetfin{X}$  the set of initial parts of $\asetf{X}$, that is:
            $$\asetfin{X}:= \bigcup_{n \in \ene} \{(u_0, u_1, ..., u_n) \in (\nnormb_X)^{n+1}: \exists (w_i)_i \in \asetf{X} \text{ s.t. } w_i = u_i, \text{ for } 0 \leq i \leq n \}.$$
    \end{itemize}
\end{defi}

\begin{obs}\label{asetfinclopen}
    \begin{itemize}
        \item[$(i)$]  Notice that an admissible set depends on the set of blocks that has been settled for $E$.
        \item[$(ii)$] Since $(\nnormb_{E})^{i}$ is a discrete topological space, the set $\asetfin{E} \cap (\nnormb_E)^{i}$ is a clopen subset of $(\nnormb_E)^{i}$, for every $i \geq 1$. 
        \item[$(iii)$]  If $X$ and $Y$ are $\nnormb_E$-block subspaces such that $Y \subseteq X$, then $\asetf{Y} \subseteq \asetf{X}$.
    \end{itemize}
\end{obs}

\begin{defi}\label{densa}
     Let $\nnormb_E$ be a set of blocks for $E$ and $\asetf{E}$ be an admissible set for $E$. We say that the pair $(\nnormb_E, \asetf{E})$ is an admissible system of blocks for $E$ if, and only if, $\nnormb_E$ and $ \asetf{E}$ satisfy the following relation:
    For every  $\nnormb_E$-block subspace $X$ of $E$, for every sequence $(\delta_n)_n$ with $0<\delta_n<1$, and $K \geq 1$, there is a collection $(A_n)_n$ of non-empty subsets of $\nnormb_X$  such that
        \begin{enumerate}
            \item[a)]  For each $n$ and for each $d \in \enefin$ such that there is $w \in \nnormb_X$ with $\supp{X}{w}= d$, we have that there are finitely many vectors $u \in A_n$ such that $\supp{X}{u} = d$.
            \item[b)] For every sequence $\wi \in \asetf{X}$ satisfying $1/K \leq  \normm{w_i} \leq   K$, for every $i$, there is  $(u_i)_i \in \asetf{X}$ such that for each $n$ we have
            \begin{itemize}
                \item[b.1)] $u_n \in A_n$,
                \item[b.2)] $\supp{X}{u_n} \subseteq \supp{X}{w_n}$,
                \item[b.3)] $\normm{w_n - u_n}< \delta_n$.
            \end{itemize}
        \end{enumerate}
\end{defi}

\subsection{Properties of admissible sets}\label{subsec:propadmsets}

\begin{prop}\label{equivc}
    Let $\nnormb_E$ be a set of blocks and $ \asetf{E}$ be an admissible set for $E$, then the following are equivalent:
     \begin{itemize}
         \item[$(i)$] For every $\baseyn \in \asetf{E}$ and every $\nnormb_E$-block subspace  $X=[x_n]_n$ we have that if $\baseun \in {(\nnormb_X)}^\omega$, then
        $$ \baseun \in \asetf{E} \iff \baseun \ast_X \baseyn  \in \asetf{E}. $$
         \item[$(ii)$] For every $\baseyn$ and $\basezn$ in $\asetf{E}$ we have that if $\basewn \in (\nnormb_E)^\omega$ then
        $$\basewn \ast_E \baseyn  \in \asetf{E} \iff \basewn \ast_E \basezn \in \asetf{E}. $$
         \item[$(iii)$] For every $\baseyn$ and $\basezn  $ in $\asetf{E}$ and every $\nnormb_E$-block subspace  $X=[x_n]_n$ we have that if $\baseun \in {(\nnormb_X)}^\omega$ then
        $$\baseun \ast_X \baseyn  \in \asetf{E} \iff \baseun \ast_X \basezn \in \asetf{E}. $$
     \end{itemize}
\end{prop}
\begin{proof}
    It follows directly from Definition \ref{setofblocks} and Definition \ref{admpro}.
\end{proof}

We can easily  prove that a set of blocks satisfies the following heredity properties:
\begin{prop}\label{subsetofblocks}
      Let $\nnormb_E$ be a set of blocks and $ \asetf{E}$ be an admissible set for $E$. If $X=[x_n]_n$ is a $\nnormb_E$-block subspace, then we have
    \begin{itemize}
        \item[$(i)$] $\nnormb_X\subseteq \nnormbuno_X$.
        \item[$(ii)$] $\{x_n: n \in \ene\} \subseteq \nnormb_X$. 
        \item[$(iii)$] If $u \in \nnormb_X$, then $\frac{u}{\|u\|} \in \nnormb_X$.
        \item[$(iv)$]  For every $\baseun \in (\nnormb_X)^\omega$ and $\basevn \in (\nnormb_{E})^\omega$, we have $ \baseun \ast_X \basevn \in (\nnormb_{E})^\omega.$ In particular, if $\basevn \in (\nnormb_{X})^\omega$, then $ \baseun \ast_X \basevn \in (\nnormb_{X})^\omega.$
        \item[$(v)$] Let $(y_i)_{i=0}^{n} \in (\nnormb_X)^{n+1} $ with $y_i < y_{i+1}$ for every $0\leq i \leq n$, and $Y=[y_i]_{i\leq n}$. If $ u \in  \nnormb_X $ is such that 
        $$u = \sum_{i =0}^{n} \lambda_{i} y_i,$$
        then 
         $$v = \sum_{i =0}^{n} \lambda_{i} x_i \in \nnormb_X.$$
    \end{itemize}
\end{prop}
\begin{proof}
    It follows directly from the definition of a set of blocks.
\end{proof}
The last heredity property is also valid for an admissible set:
\begin{prop}\label{admblock}
     Let  $\nnormb_E$ be a set of blocks and $ \asetf{E}$ be an admissible set for $E$.  Let $X = [x_n]_n$ be a  $\nnormb_E$-block subspace. The set $\asetf{X}$ satisfies the following properties:
     \begin{itemize}
        \item[$(i)$] $\asetf{X}$ is a closed subset of $\nnomega{X}$. 
        \item[$(ii)$] Any block basis $\baseyn$ in $\nnomega{X}$ belongs to $\asetf{X}$. 
        \item[$(iii)$] For every $\basevn \in \asetf{X}$ and every $\nnormb_X$-block subspace  $Y=[y_n]_n$ we have that if $\baseun \in {(\nnormb_Y)}^\omega$, then
        $$ \baseun \in \asetf{X} \iff \baseun \ast_Y \basevn  \in \asetf{X}. $$
        \item[$(iv)$]  Let $Y = [y_n]_n$ be a $\nnormb_X$-block subspace. For every $\baseun \in \asetf{X}$ and $k \in \ene$, there is $\basevn \in Y^\omega$ such that $(u_0, ..., u_k, v_0, v_1, ...) \in \asetf{X}$.
    \end{itemize}
\end{prop}

\begin{proof}
    It follows directly from Definition \ref{admpro}.
\end{proof}

Notice that if $X$ is a $\nnormb_E$-block subspace, then using $ii)$ in Proposition \ref{admblock} we conclude that $\asetfin{X}$ is infinite. If $(\nnormb_E, \asetf{E})$ is an admissible system of blocks for $E$ and $X$ is a $\nnormb_E$-block subspace, then as a consequence of Proposition \ref{subsetofblocks} and Proposition \ref{admblock}, the pair $(\nnormb_X,\asetf{X})$ can be thought as an ``admissible subsystem of blocks'' relative to $X$.  The ``relativization'' of the condition given in Definition \ref{densa} to $X$ is clearly true: 
    For every  $\nnormb_X$-block subspace $Y$ of $X$, for every sequence $(\delta_n)_n$ with $0<\delta_n<1$, and $K \geq 1$, there is a collection $(A_n)_n$ of non-empty subsets of $\nnormb_Y$  such that
        \begin{enumerate}
            \item[a)]  For each $n$ and for each $d \in \enefin$ such that there is $w \in \nnormb_Y$ with $\supp{Y}{w}= d$, we have that there are finitely many vectors $u \in A_n$ such that $\supp{Y}{u} = d$.
            \item[b)] For every sequence $\wi \in \asetf{Y}$ satisfying $1/K \leq \min_{i} \normm{w_i} \leq \sup_{i} \normm{w_i} \leq K$, there is  $(u_i)_i \in \asetf{Y}$ such that for each $n$  such that
                b.1) $u_n \in A_n$,
                b.2) $\supp{Y}{u_n} \subseteq \supp{Y}{w_n}$,
                b.3) $\normm{w_n - u_n}< \delta_n$.
        \end{enumerate}

\begin{prop}\label{propadmpro}
    Let $\nnormb_E$ be a set of blocks for $E$ and $\asetf{E}$ be an admissible set for $E$. The following statements are true:
    \begin{itemize}
        \item[$(i)$]  Let $X$ be a $\nnormb_E$-block subspace. If $\ui \in (\nnormb_X)^\omega$ satisfies that for every $n \in \ene$ the finite sequence $(u_i)_{i=0}^n \in \asetfin{E}$, then $\ui \in \asetf{X}$. 
        
        \item[$(ii)$] If $X=[x_n]_n$ and $Y=[y_n]_n$ are $\nnormb_E$-block subspaces such that $\basexn \sim \baseyn$, and $T: X \rightarrow Y$ is the linear map such that $\forall n \in \ene \: (T(x_n) = y_n)$, then 
        $T(\asetf{X})= \asetf{Y}.$
        
        \item[$(iii)$]  If $X$ is a $\nnormb_E$-block subspace, then
        $\asetfin{X} = \asetfin{E} \bigcap \; \bigcup_{i\geq 1} X^i$.
    \end{itemize}
\end{prop}

\begin{proof}
    \begin{itemize}
        \item[$(i)$] Let $X$ and $u = \ui \in (\nnormb_X)^\omega $ be as in the hypothesis. For each $n \in \ene$ let $v^n = (v_{i}^n)_i \in \asetf{E}$ such that $u_i = v_{i}^n$ for every $0 \leq i \leq n$. Without loss of generality, we can suppose each $v^n \in \asetf{X}$ (using $d)$ in Definition \ref{admpro} we can find a sequence in $\asetf{X}$ which coincide with $v^n$ in the first $n$ coordinates). Thus, $v_{j}^n = u_j$,  for every $n \geq j$. Which means that for each $j \in \ene$ we have $(v_{j}^n) \tends{n}{\infty}{u_j}$ in $\nnormb_X$. Therefore, $v^n \tends{n}{\infty}{u}$ in  $\nnomega{X}$. Using $(i)$ in Proposition \ref{admblock}, $u \in \asetf{X}$.
        
        
        \item[$(ii)$]
        
        Let $X=[x_n]_n$ and $Y=[y_n]_n$ be $\nnormb_E$-block subspaces of $E$, and let $T: X \rightarrow Y$ be as in the hypothesis. Notice that by $b)$ in Definition \ref{admpro}, $\basexn, \baseyn, \baseen \in \asetf{E}$.

        Let $\baseun \in \asetf{X}$ with  
        $$u_n = \sum_{i \in \supp{X}{u_n}} \lambda_{i}^n x_i,$$
        for each $n \in \ene$. We want to show that $ (T(u_n))_n \in \asetf{Y}$. 
        Notice that $ (T(u_n))_n = \baseun \ast_X \baseyn $, then by $c)$ in Definition \ref{admpro}, $(T(u_n))_n \in \asetf{E} \cap Y^\omega = \asetf{Y}$. 
        
        On the other hand, let $\basevn \in \asetf{Y}$ with
        $$v_n = \sum_{i \in \supp{Y}{v_n}} \alpha_{i}^n y_i,$$
        for every $n \in \ene$. Set for each $n \in \ene$
         $$u_n = \sum_{i \in \supp{Y}{v_n}} \alpha_{i}^n x_i.$$
        Clearly, $T(u_n) = v_n$ for every $n$ and  $\baseun = \basevn \ast_Y \basexn$. By $c)$ of Definition \ref{admpro} $\baseun \in \asetf{E} \cap X^\omega = \asetf{X}$.

        \item[$(iii)$] Let $X=[x_n]_n$ be a $\nnormb_E$-block subspace. Since $ \asetf{E}\cap X^\omega = \asetf{X}$, it follows that
        $$\asetfin{X} \subseteq \asetfin{E} \bigcap \; \bigcup_{i\geq 1} X^i .$$
        
        Suppose that $(u_i)_{i=0}^n \in \asetfin{E} \cap X^{n+1}$, for some $n \in \ene$. Using $d)$ in Definition \ref{admpro}, there is $(u_i)_{i = n+1}^\infty \in X^\omega$, such that $u = (u_0, ..., u_n, u_{n+1}, ...) \in \asetf{X}$. Then, $(u_i)_{i=0}^n \in \asetfin{X}$.
    \end{itemize}
\end{proof}

\subsection{Admissible families}\label{subsec:admissiblefamilies}

\begin{defi}
    We define the operation $\circledast  : \partes(\omega)^{\omega} \times \partes(\omega)^{\omega} \rightarrow \partes(\omega)^{\omega} $, as follows: given $U = (U_i)_i$ and $V=(V_i)_i$ in $\partes(\omega)^{\omega}$, we define $U \circledast V = (W_i)_i $ as $W_i = \cup_{j \in U_i}V_j$, for every $i \in \ene$. 
\end{defi}

\begin{defi}
    \begin{itemize}
        \item[$(i)$] We denote by $bb$ the set of sequences of successive non-empty finite subsets of $\ene$, that is
    $$bb := \{(U_i)_i \in {\rm FIN}^\omega:  \forall i \in \ene \;(U_i < U_{i+1})\}.$$
    
    \item[$(ii)$] We denote by $db$ the set of sequences of non-empty finite subsets of $\ene$ whose elements are mutually disjoint:
    $$db(\ene) := \{(U_i)_i \in {\rm FIN}^\omega: \forall i\neq j\;(U_i \cap U_{j} = \emptyset)\}.$$
    \end{itemize}
\end{defi}

\begin{obs}
    Observe that:
    \begin{itemize}
        \item[$(i)$] the operation $\circledast$ is internal on ${\rm FIN}^\omega$, $bb$ and $db$ respectively. 
        \item[$(ii)$]  If $U = (U_i)_i$ and $V=(V_i)_i$ in $\partes(\omega)^{\omega}$ and  $U \circledast V = (W_i)_i $, then 
    $$\bigcup_{i \in \ene} W_i \subseteq \bigcup_{i \in \ene} V_i.$$ 
        \item[$(iii)$] Note that $e:= (\{i\})_{i}$ is a neutral element for operation $\circledast$, that is, if $U \in \partes(\omega)^{\omega} $, then $U \circledast e = e \circledast U = U$. 
    \end{itemize}
\end{obs}

We shall consider ${\rm FIN}^\omega$ as a topological subspace of $(2^\omega)^\omega$, where $(2^\omega)^\omega$ is endowed with the product topology which results from considering $2^\omega$ as the Cantor space with its topology. The following proposition follows directly  from the definition of operations $\ast_X$ and $\circledast$.

\begin{prop}\label{obssupport}
     Let $\nnormb_E$ be a set of blocks for $E$. Let $X$ be a $\nnormb_E$-block subspace of $E$.
     Suppose $\baseun \in \nnomega{X}$ and $\basevn \in \nnomega{E}$. If 
     $\basewn = \baseun \ast_X \basevn$, then
    \begin{equation}
        (\supp{E}{w_n})_n = (\supp{X}{u_n})_n \circledast (\supp{E}{v_n})_n.
    \end{equation}
    Also, if $\basevn$ is a basic sequence, then for each $n$
    \begin{equation}
        \supp{[v_i]_i}{w_n} = \supp{X}{u_n}.
    \end{equation}
\end{prop}

\begin{defi}\label{admset}
    We say that a non-empty subset $\admfam\subseteq {\rm FIN}^\omega$  is an admissible family if, and only if,  the following conditions  are satisfied:
    \begin{itemize}
        \item[a)] $\admfam$ is a closed subset of ${\rm FIN}^\omega$.
        \item[b)] $bb \subseteq \admfam$. 
        \item[c)] For every $(U_i)_i, (V_i)_i  \in \admfam$ and every $(W_i)_i \in {\rm FIN}^\omega$, we have
        \begin{equation}
            (W_i)_i \circledast (U_i)_i \in \admfam \iff  (W_i)_i \circledast (V_i)_i  \in \admfam.
        \end{equation}
        \item[d)] For every $(U_i)_i, (V_i)_i \in \admfam$ and $n \in \ene$, there is $(\{t_i\})_i $ subsequence of $e$ such that
        $$(U_0, U_1, ..., U_n, W_0, W_1, ...) \in \admfam,$$
        where $(W_i)_i = (t_i)_i \circledast (V_i)_i$. 
        
    \end{itemize}
\end{defi}

\begin{obs}\label{equivadmset}
\begin{itemize}
    \item[$(i)$] If $\admfam$ is an admissible set, condition $b)$ implies that the neutral element $e $ belongs to $\admfam$. 
    \item[$(ii)$] It is easy to see that the condition $c)$ in Definition \ref{admset} is equivalent to the following statement: for every $(V_i)_i  \in \admfam$ and every $(W_i)_i \in {\rm FIN}^\omega$, we have
        \begin{equation}
            (W_i)_i \in \admfam \iff  (W_i)_i \circledast (V_i)_i  \in \admfam.
        \end{equation}
\end{itemize}
\end{obs}

\begin{prop}\label{exadmset}
    The sets ${\rm FIN}^\omega$, $bb$ and $db$ are admissible families.
\end{prop}

\begin{proof}
    It is clear that ${\rm FIN}^\omega$ is an admissible set.
    $bb$ and $db(\ene)$ satisfy condition $a)$ of Definition \ref{admset} as a consequence of the topology we have considered on  ${\rm FIN}^\omega$. Conditions $b)$, $c)$ and $d)$ in Definition \ref{admset} are consequence of the properties of the operation $\circledast$ and the fact that sequences of the type $(\{m+i\})_i$ are in $ bb$, and, therefore in $  db(\ene)$. 
\end{proof}

\begin{prop}\label{supportperm}
    The set 
    $$per := \{(U_i)_i \in {\rm FIN}^\omega : \exists \pi \text{ a permutation of } \ene \text{ s.t. } \forall i \in \ene (U_{\pi(i)}<U_{\pi(i+1)}) \}$$
    is not an admissible family.
\end{prop}
\begin{proof}
    Consider 
     $ U :=  (\{0,1\}, \{2\}, \{3\},  ...)$ and $V :=(\{0\}, \{2\}, \{1\}, \{3\}, \{4\}, ...)$ both in $per(\ene)$.
    Notice that $U = U \circledast e$ and $V$ belong to $per(\ene)$, but 
    $$U \circledast V = (\{0,2\}, \{1\}, \{3\}, \{4\}, ...)$$ does not. Then, $per(\ene)$ fails to satisfy condition $c)$ in Definition \ref{admset}.
\end{proof}
 The next definition establishes that an admissible family determines an admissible set for $E$.
\begin{prop}\label{admsetadmprop}
    Let $\admfam$ be an admissible family. Let $\nnormb_E$ be a set of blocks for $E$. 
    Define the set $\admfam(\nnormb_E)$ as follows:
    for any $\ui \in \nnomega{E}$,
     $\ui \in \admfam(\nnormb_E)$ if, and only if,  $(\supp{E}{u_i})_i \in\admfam$.
    
    Then, $\admfam(\nnormb_E)$ is an admissible set for $E$.
\end{prop}

\begin{proof}
    Suppose $\admfam$, $\nnormb_E$ be as in the hypothesis. Define
    \begin{equation}
        \aset := \admfam(\nnormb_E) = \{\ui \in \nnomega{E} :(\supp{E}{u_i})_i \in\admfam \}.
    \end{equation}

    Let us check each condition of Definition \ref{admpro}.
    \begin{itemize}
        
        \item[a)] Suppose $v:=\vi \in \overline{\aset}\subseteq \nnomega{E}$ and let $\ui$ be a sequence in $(\aset)^\omega$ which converges to $v$. Then, if for each $i$,  $u_i = (u_{j}^i)_j$, then $u_{j}^i \tends{i}{\infty}{v_j}$  in $(\nnormb_{E})^\omega$, for every $j \in \ene$. Thus, for each $j \in \ene$ there is $N_j>0$ such that $u_{j}^i = v_j$ (in particular $\supp{E}{u_{j}^i} = \supp{E}{ v_j}$), for every $i > N_j$. This means that for each $j \in \ene$, 
        \begin{equation}\label{eqaux2}
            \supp{E}{u_{j}^i} \tends{i}{\infty}{\supp{E}{ v_j}}\;\; \text{ in } {\rm FIN}.
        \end{equation}
        
        For each $i \in \ene$, $u_i \in \aset \Rightarrow U_i := (\supp{E}{u_{j}^i})_j \in\admfam$. Equation \eqref{eqaux2} shows that $(U_i)_i$ converges to $ (\supp{E}{v_{j}})_j \in {\rm FIN}^\omega$. Since $\admfam$ is closed in ${\rm FIN}^\omega$, $ (\supp{E}{v_{j}})_j \in \admfam$. By the definition of $\aset$, this means that $v \in \aset$. 
        
        \item[b)] Let $\baseyn$ be a sequence of successive blocks, that is $\forall n \in \ene \;(y_n \in \nnormb \: \& \: y_n < y_{n+1})$. Then, $(\supp{E}{y_i})_i \in bb(\ene)$. By item $b)$ in Definition \ref{admset}, $bb(\ene) \subseteq \admfam$, so $\baseyn \in \aset$. 
        
        \item[c)]
        
        Let $\baseyn \in \aset$ and $X=[x_n]_n$ be a $\nnormb_E$-block subspace.  Suppose $\baseun \in {(\nnormb_X)}^\omega$, where for each $n \in \ene$,
        
        $$ u_n = \sum_{i \in \supp{X}{u_n} }\lambda_{i}^n x_i.$$
        
        We want to see that 
        \begin{equation}\label{eq19}
            \baseun \in \aset \iff \basevn:= \baseun \ast_X \baseyn \in \aset
        \end{equation}
        Observe that  $\baseun \in {(\nnormb_X)}^\omega$ and,  due to $(iv)$ in Proposition \ref{subsetofblocks}, we know that $\baseun \ast_X \baseyn  \in \nnomega{E}$.
        
       By Proposition \ref{obssupport}, we know that
        \begin{equation}
           (\supp{E}{v_n})_n = (\supp{X}{u_n})_n \circledast (\supp{E}{y_n})_n.
       \end{equation}
       As a consequence of the last equation, the definition of $\aset$ and condition $c)$ of Definition \ref{admset}, we obtain
        \begin{eqnarray*}
           \baseun \in \aset &\iff& (\supp{E}{u_n})_n \in \admfam\\
            &\iff&  (\supp{X}{u_n})_n \circledast (\supp{E}{x_n})_n\in \admfam\\
            &\iff&(\supp{X}{u_n})_n \circledast (\supp{E}{y_n})_n\in \admfam\\
            &\iff&(\supp{E}{v_n})_n\in \admfam\\
            &\iff&\basevn\in \aset
        \end{eqnarray*}
      
        \item[d)]  Let $\baseyn$ a $\nnormb_E$-block sequence and $Y=[y_n]_n$. By using item $b)$, we have $(\supp{E}{y_n})_n \in \admfam$. Let $(u_i)_i \in \aset$, so $(\supp{E}{u_i})_i \in \admfam$. By condition $d)$ in Definition \ref{admset} there is $(\{a_i\})_i \in bb(\ene)$ such that 
        \begin{equation}\label{eqaux12}
             (\supp{E}{u_0}, \supp{E}{u_1}, ..., \supp{E}{u_n}, B_0, B_1, ...) \in \admfam,
        \end{equation}
        where $(B_i)_i = (\{a_i\})_i \circledast (\supp{E}{y_i})_i$. For each $i \in \ene$, let $z_i = y_{a_i}$.
        It is clear that $(z_i)_i \in {\nnormb_{Y}}^\omega$ and $\supp{E}{z_i} = B_i$, for every $i \in \ene$. Then, by Equation \eqref{eqaux12} we have
        $$(u_0, ..., u_n, z_0, z_{1}, ...) \in \aset.$$
    \end{itemize}
\end{proof}


    Under the hypothesis of Proposition \ref{admsetadmprop}, we shall refer to the obtained set $\admfam(\nnormb_E)$ as the admissible set for $E$ determined by the admissible family $\admfam$.

\begin{prop}
    Let $\admfam$ be an admissible family. Let $\nnormb_E$ be a set of blocks for $E$. Let $X$ be a $\nnormb_E$-block subspace of $E$. If $\asetf{E}=\admfam(\nnormb_E)$, then $\asetf{X}=\admfam(\nnormb_X)$.
\end{prop}
\begin{proof}
    It follows from the facts that $\asetf{X}=\asetf{E} \cap X^\omega$, that $\nnormb_X= \nnormb_E \cap X$, and that for every $\baseun \in \nnomega{X}$,
    \begin{equation}\label{eq21}
        (\supp{E}{u_n})_n \in \admfam \iff (\supp{X}{u_n})_n \in \admfam.
    \end{equation}
    And this last fact follows from Proposition \ref{obssupport}.
\end{proof}


From  Proposition \ref{admsetadmprop} we obtain immediately:

\begin{prop}\label{exadmprop}
    Let $\nnormb_E$ be a set of blocks for $E$. The following sets are admissible for $E$:
    \begin{itemize}
        \item[$(i)$] The set $(\nnormb_E)^\omega$ of infinite sequences of $\nnormb_E$-blocks.
        \item[$(ii)$] The set $bb(\nnormb_E)$ of $\nnormb_E$-block sequences of $E$.
        \item[$(iii)$] The set $db(\nnormb_E)$ of infinite sequences of pairwise disjointly supported $\nnormb_E$-blocks.
    \end{itemize}
\end{prop}


\section{Embeddings and minimality}\label{sec:embedding}

In this section we shall use the previous sets of blocks and admissible sets to code different kinds of embeddings. Doing this we shall be able of associate to each embedding a notion of tightness and of minimality, which in some cases, coincide with minimal notions studied previously, for example in \cite{Ferergodic}. To simplify the notations we shall fix a Banach space $E$ with normalized basis $\baseen$. 

\begin{defi}\label{defiAembed}
    Let $\nnormb_E$ be a set of blocks for $E$ and $\asetf{E}$ an admissible set for $E$. Suppose that $X$ is a $\nnormb_E$-block subspace. Let $Y$ be a Banach space with normalized basis $\baseyn$ and suppose $K\geq 1$. 
    \begin{itemize}
        \item[$(i)$]  We shall say that $Y$ $\asetf{X}$-embeds in $X$ with constant $K$ (in symbols $Y \embast_K X$) if, and only if, there is some sequence $\baseun \in \asetf{X}$  of blocks such that $\baseun \sim_K \baseyn$.
        \item[$(ii)$] We say that $Y$ $\asetf{X}$-embeds
        in $X$  (in symbols $Y \embast X$), if $Y \embast_K X$ for some constant $K \geq 1$. 
    \end{itemize}
\end{defi}
A number of natural properties follow directly from Definition \ref{defiAembed} and will be used. For example, the definition guarantees that 
if $Y$ is a $\nnormb_X$-block subspace of $X$ and $Z \embast Y$, then $Z \embast X$ as well.

\begin{defi}\label{defispaceminimal}
    Let $E$ be a Banach space with normalized basis $\baseen$. Let $\nnormb_E$ be a set of blocks for $E$ and $\asetf{E}$ an admissible set for $E$. Suppose that $X$ is a $\nnormb_E$-block subspace. We say that $X$ is $\asetf{E}$-minimal  if, and only if, for every $\nnormb_X$-block subspace $Y$ we have that $X \embast Y$.
\end{defi}

The following proposition establishes that the property of being $\asetf{E}$-minimal is hereditary by taking $\nnormb_E$-subspaces.

\begin{prop}
     Let $E$ be a Banach space with normalized basis $\baseen$. Let $\nnormb_E$ be a set of blocks for $E$ and $\asetf{E}$ be an admissible set for $E$. Suppose that $X$ is a  $\nnormb_E$-block subspace which is $\asetf{E}$-minimal. If $Y$ is a $\nnormb_X$-block subspace of $X$, then $Y$ is $\asetf{E}$-minimal.
\end{prop}

\begin{proof}
    Let $E$, $\nnormb_E$, $\asetf{E}$ and $X$ be as in the hypothesis. Let $Y=[y_n]_n$ be a $\nnormb_X$-block subspace of $X$. Let $Z=[z_n]_n$ be a $\nnormb_Y$-block subspace of $Y$ (so it is also a $\nnormb_X$-block subspace of $\basexn$). We want to see that $Y\embast Z$.  
    
    By the $\asetf{E}$-minimality of $X$, we have $X\embast Z $, thus there is $\baseun \in \asetf{Z} \subseteq \asetf{X} $ such that $\basexn \sim \baseun$.
    By $(iii)$ in Proposition \ref{admblock} we have
    $$ \baseyn \in \asetf{X} \Rightarrow \basewn := \baseyn \ast_X \baseun \in \asetf{X} \cap Z = \asetf{Z}. $$
    
    Then, $\basewn$ is a block basis of the basic sequence $\baseun$ of $\nnormb_Z$-blocks (it is not necessarily  a block sequence of $X$ because $\baseun$ need not be a block sequence).  Since $\baseun \sim \basexn$ and each $w_n$ has the same scalars in it expansion than $y_n$, we have that $\baseyn \sim \basewn$. So, $Y \embast Z$.
\end{proof}

\begin{obs}\label{obsminimalimpli}
Notice that, in the context of Proposition \ref{exadmprop}, for a fixed set of blocks $\nnormb_E$, we have 
\begin{center}
    $bb({\nnormb_E})$-minimality $\implies$ $db({\nnormb_E})$-minimality $\implies$ $ \nnomega{E}$-minimality.
\end{center}
\end{obs}

\section{Interpretations for the set of blocks}\label{sec:blocksubspaceinF}

Depending on the set of blocks $\nnormb_E \subseteq \nnormbuno_E$  we have chosen for the Banach space $E$, it is possible to give different interpretations for the admissible set considered. In this subsection we shall explore various sets of blocks and analyze the admissible sets obtained in Proposition \ref{exadmprop} in each context.

\subsection{Blocks as the non-zero F-linear combinations}\label{blocksbeingeverything}

We shall start this exposition with the biggest set of blocks possible. Consider the set of blocks $\nnormbuno_E$, that is the set which elements are all non-zero finitely supported $\mathbf{F}_E$-linear combinations of the basis $\baseen$. This set of blocks coincide with the context of blocks used by A. Pelczar in \cite{pelczar} and also by V. Ferenczi and Ch. Rosendal in \cite{FerencziRosendal1}. 

In this context, a $\nnormbuno_E$-block sequence is a block basis which elements are non-zero finitely supported $\mathbf{F}_E$-linear combinations and
$$bb_{\nnormbuno}(E) = \{(x_n)_n \in (\nnormbuno_E)^\omega: \forall n \in \ene ( x_n < x_{n+1} \; \& \; \normm{x_n} = 1)\}.$$

\begin{obs}
    Any normalized finitely supported basic sequence $\baseyn$ in $E = [e_n]_n$, is equivalent to $\basezn \in (\nnormbuno_E)^\omega$ with $\supp{E}{z_n} = \supp{E}{y_n}$, for every $n$. This is a consequence of the density of $\nnormbuno_E$ in $E$ and the principle of small perturbations.
\end{obs}
The proof of the following proposition is an adaptation of the beginning of the proof of Lemma 3.7 in \cite{FerencziRosendal1}.
\begin{prop}\label{admissibleblockDD}
    Suppose that we are considering the set of blocks for $E$ as $\nnormbuno_E$ and that $\asetf{E}$ is the admissible set for $E$ determined by an admissible family $\admfam$, i.e., $\asetf{E}=\admfam(\nnormbuno_E)$. Then, the pair $(\nnormbuno_E, \asetf{E})$ is an admissible system of blocks for $E$. 
    
\end{prop}
\begin{proof}

    Let $X=[x_n]_n$ be a $\nnormbuno_E$-block subspace, $(\delta_n)_n$ with $0<\delta_n<1$ and $K\geq 1$.
    We are going to construct for each $n \in \ene$ sets $D_n$ of not necessarily normalized $\nnormbuno_X$-blocks with the following properties:
    \begin{enumerate}
        \item For each $d \in \enefin $, there are a finite number of vectors $u \in D_n$ such that $\supp{X}{u} = d$.
        \item If $w$ is a $\nnormbuno_X$-block vector with norm in $[\frac{1}{K}, K]$, then there is some $u \in D_n$, with the same support in $X$ of $w$ such that $\normm{w-u}< \delta_n$.
    \end{enumerate}
    
    Before the proof of the existence of such sets $D_n$, let us show why this is sufficient: let $\vi \in \asetf{X}$ satisfying $\frac{1}{K}\leq \normm{v_i}\leq K$, for every $i \in \ene$. Since $\vi \in \asetf{X}$ and $\asetf{E}$ is the admissible set for $E$ determined by an admissible family $\admfam$, it follows that 
    \begin{equation}\label{eqaux11}
        (\supp{X}{v_i})_i \in \admfam.
    \end{equation}
    Using $2)$, for each $i$ there is $w_i \in D_i$ with  $\normm{w_i-v_i}< \delta_i$ and $\supp{X}{w_i} = \supp{X}{v_i}$, so by Equation \eqref{eqaux11} $(\supp{X}{w_i})_i \in \admfam$, what means that $\wi \in \asetf{X}$. Therefore, $(\nnormbuno_E, \asetf{E})$ is an admissible system of blocks for $E$.
    
    Let us prove that such sets $D_n$ exist: set $n \in \ene$. We proceed by induction: if $d\in [\ene]^{1}$, then, since the closed K-ball of $[x_i]_{i \in d}$ is totally bounded and $\nnormbuno_E$ is dense in $E$, it is possible to find a finite  $U_d = \{u_{1}^d, ... u_{m(d)}^d\}\subset \overline{\ba}_{K}([x_i]_{i \in d}) \cap \nnormbuno_E $ such that if $w \in [x_i]_{i \in d}$ and $\frac{1}{K}\leq \normm{w} \leq K$, then there is some $j \leq m(d)$ with $\normm{w - u_{j}^d}< \delta_n$.
    
    Suppose we have found for every $d \in [\ene]^{<m}$ such vectors  $U_d= \{u_{1}^d, ... u_{m(d)}^d\}\subset \overline{\ba}_{K}([x_i]_{i \in d}) \cap \nnormbuno_E$ with the desired property. Let $d \in [\ene]^m$, then as the closed K-ball of 
    $$[x_i]_{i \in d} \setminus \bigcup_{d' \subset d} [x_i]_{i \in d'}$$ 
    is again totally bounded and $\nnormbuno_E$ is dense in $E$, there is $U_d= \{u_{1}^d, ... u_{m(d)}^d\}\subset \overline{\ba}([x_i]_{i \in d}) \cap \nnormbuno_E$ such that if $w \in [x_i]_{i \in d}$,  $\frac{1}{K}\leq \normm{w} \leq K$ and $\supp{}{w} = d$, then there is some $j \leq m(d)$ such that $\normm{w - u_{j}^d}< \delta_n$.
    Finally, set 
    $$D_n = \bigcup_{d \in \enefin} U_d.$$ 
\end{proof}

As an immediate consequence of Proposition \ref{exadmprop} and Proposition \ref{admissibleblockDD}, we deduce:

\begin{coro}\label{admisiDD3}
    The pairs $(\nnormbuno_E, \nnormbuno_E^\omega)$,  $(\nnormbuno_E, db(\nnormbuno_E)$ and $(\nnormbuno_E, bb(\nnormbuno_E)$ are admissible systems of blocks for $E$. 
\end{coro}

Notice that it is a fact frequently used, for example in \cite{pelczar} and \cite{FerencziRosendal1}, that after a perturbation argument,  an $(\nnormbuno_E)^\omega$-embedding is ``equivalent'' to the usual isomorphic embedding, i.e. if $X=[x_n]_n$ is a $\nnormbuno_E$-block subspace, then $Y \embast X \iff Y \hookrightarrow X$, when $\asetf{E} = (\nnormbuno_E)^\omega$. Furthermore, if $Y \hookrightarrow_K X$ for some $K \geq 1$, then for any $\varepsilon >0$ we have $Y \embast_{K+\varepsilon} X$.

As it was proved in Proposition \ref{supportperm}, the family $per$ is not admissible for ${\rm FIN}^\omega$. So, Proposition \ref{admsetadmprop} can not be used to determined whether  the set of sequences of blocks that are a permutation of a block basis is an admissible set for the Banach space $E$. In the next proposition we actually prove that such set of sequences is not admissible for $E$.

\begin{prop}\label{remperm1}
    The set 
    $$per(\nnormbuno_E):= \{\basexn \in (\nnormbuno_E)^\omega : (\supp{E}{x_n} )_n \in per\}$$
    is not admissible for $E$.
\end{prop}
\begin{proof}
    Let 
    \begin{equation*}
        \basezn = (e_0, e_2, e_1, e_3, e_4, ...) \text{ and } \basewn = (e_0 + e_1, e_2, e_3, e_4, ...)
    \end{equation*}
   Both $\basezn$ and $\basewn$ are permutations of  $\nnormbuno_E$-block sequences but 
    $\basewn \ast_E \basezn = (e_0 + e_2, e_1, e_3, ...)$ is not.  So, condition $c)$ in Definition \ref{admpro} is not satisfied.
\end{proof}

\subsection{Blocks as the set of vectors of the basis}

The smallest set of blocks we can consider is the set for which  the blocks are exclusively the vectors of the basis $\nnormbasis_E$.
Notice that in this case all blocks are normalized.
In this context a $\nnormbasis_E$-block sequence is a subsequence of the basis, and a sequence of disjointly supported blocks is a sequence of different elements of the basis (not necessarily in increasing order). 

\begin{prop}\label{sequencecompadmprop}
    Let $\admfam$ be an admissible family. Then   $(\nnormbasis_E, \admfam(\nnormbasis_E))$ is an admissible system of blocks for $E$. 
\end{prop}
\begin{proof}
    It follows directly from the fact that for each $n \in \ene$ only one $\nnormbasis_E$-block has support $\{n\}$. In this case,  the conditions asked in Definition \ref{densa} are trivial. What we are saying is that for the case of embedding, minimality or tightness by sequences, it is not necessary to perturb the vectors along the proofs.  
\end{proof}

As a corollary, using the admissible families $bb$ and $db$ (Proposition \ref{exadmprop}), we obtain:



\begin{coro}
     The pairs $(\nnormbasis_E, bb(\nnormbasis_E))$ and $(\nnormbasis_E, bb(\nnormbasis_E))$ 
     are admissible system of blocks for $E$, corresponding to the admissible sets of subsequences of $(e_n)_n$ and of pairwise distinct elements of $(e_n)_n$, respectively.
\end{coro}

\begin{obs}\label{obsfinblocks}
    Let $\nnormb_E$ be a set of blocks for the Banach space $E$ and $\asetf{E}$ be an admissible set determined by an admissible family. Notice that Proposition \ref{sequencecompadmprop} is true in the case where for each $d \in \enefin$ such that there is $w \in \nnormb_E$ with $\supp{E}{w} = d$, we have that the set $\{u \in \nnormb_E: d=\supp{E}{u}\}$ is finite. Under this hypothesis a pair $(\nnormb_E,\asetf{E})$ is an admissible system of blocks for $E$.
\end{obs}

\begin{defi}
    Let $Y$ be a Banach space with  normalized basis $\baseyn$. We write $(y_n)_n\embs (e_n)_n$ to denote that $(y_n)_n$ is equivalent to a subsequence of $(e_n)_n$.
\end{defi}

 Note that $(y_n)_n \embs (e_n)_n$ if, and only if, $Y \embast E$ where the set of blocks is $\nnormbasis_E$  and  $\aset_E = bb(\nnormbasis_E)$ is the admissible set for $E$.
Therefore this definition enters into the general context of our paper.

\subsection{Blocks as signed elements of the basis}

Additionally, we shall study the case of the set of blocks $\nnormbasis_{E}^{\pm} $ for $E$, where we recall that $ x \in  \nnormbasis_{E}^{\pm}$ if, and only if,  $x=\varepsilon e_k$, for some $k \in \ene$ and some sign $\varepsilon \in \{-1,1\}$.

Since for each $n \in \ene$ only two vectors $e_n$ and $-e_n$ in $\nnormbasis_E$ have as support $\{n\}$, from Remark \ref{obsfinblocks} we have immediately:

\begin{prop}\label{signedadmissible}
    Let $\admfam$ be an admissible family. Then  $(\nnormbasis_{E}^{\pm}, \admfam(\nnormbasis_{E}^{\pm}))$ is an admissible system of blocks for $E$. 
\end{prop}

\begin{defi}
    We say that $\basexn$ is a signed subsequence of $\baseen$ if 
    $$\basexn \in bb(\nnormbasis^{\pm}_E):= \{(\varepsilon_i e_{n_i})_i : (n_i)_i \in \ene^\omega \text{ is increasing } \& \; (\varepsilon_i)_i \in \{-1,1\}^\omega\}.$$
    The sequence $\basexn$ is a signed permutation of a subsequence of $\baseen$ if
$$\basexn \in db(\nnormbasis^{\pm}_E):= \{(\varepsilon_i e_{n_i})_i : (n_i)_i \in \ene^\omega \text{ are mutually distinct } \& \; (\varepsilon_i)_i \in \{-1,1\}^\omega\}.$$

\end{defi}



From Proposition \ref{admsetadmprop} and Proposition \ref{signedadmissible}, we have: 

\begin{coro}
the pairs $(\nnormbasis_{E}^{\pm}, bb(\nnormbasis^{\pm}_E))$ and $(\nnormbasis_{E}^{\pm}, db( \nnormbasis^{\pm}_E))$  are admissible systems of blocks for $E$, associated to the admissible sets of  signed subsequences of $\baseen$ and signed permutations of subsequences of $\baseen$, respectively.
\end{coro}

   

\section{Summary of types of minimality}\label{sec:typesofminimality}

We can summarize the interpretation of each embedding as follows: let $Y$ be a Banach space with  normalized basis $\baseyn$. Suppose that we are considering the set of blocks $\nnormb_E$ to be $\nnormbasis_E$, $\nnormbasis_{E}^{\pm}$ or $\nnormbuno_E$, and $\asetf{E}$ the admissible set determined by any of the admissible families ${\rm FIN}^\omega$, $bb$ or  $db$. To say that $Y \embast E$ means in each case that the basis $\baseyn$  is equivalent to a sequence $\basexn$ in $E^\omega$ which satisfies the respective condition we have represented in Table 1.

Notice that since $\baseyn$ is a basic sequence, in the trivial cases when the admissible family is ${\rm FIN}^\omega$, then, necessarily, $\basexn$ must also be basic, so in particular $x_n \neq x_m$ for $n \neq m$. For that reason, the first and third rows of the $\nnormbasis_E$ and $\nnormbasis_{E}^\pm$ columns are the same.

  \begin{center}
\begin{table}[hbt!]\label{table2}
\resizebox{\textwidth}{!}{%
     \begin{tabular}{||m{1.5cm}| m{4cm}| m{5cm} | m{4.5cm} ||}
        \hline\hline 
        \diagbox{$\admfam$}{$\nnormb_E$} & $\;\;\;\;\;\;\;\;\;\;\;\;\; \;\;\nnormbasis_E$&  $\;\;\;\;\;\;\; \;\;\;\;\;\;\;\;\;\;\; \;\nnormbasis_{E}^\pm$& $\;\;\;\;\;\;\; \;\;\;\;\;\; \;\;\;\;\nnormbuno_E$\\
        \hline\hline
        
        \multirow[c]{3}{=}
      
        & & & \\  {${\rm FIN}^\omega$  }
          & $\basexn $ is a permutation of a subsequence of $\baseen$  &   $\basexn$ is a permutation of a signed subsequence & $\basexn$ is a sequence of finitely supported vectors of $\nnormbuno_E$\\
         
        \hline \hline
        \multirow[c]{4}{=}
        & & & \\ 
        {$bb$}
          & $\basexn $ is a subsequence of $\baseen$ &  $\basexn  $ is a signed subsequence of $\baseen$ &$\basexn$ is a $\nnormbuno_E$-block sequence\\
        \hline \hline
        
        \multirow[c]{4}{=}
        & & & \\ {$db$}
          &  $\basexn $ is a per\-mu\-ta\-tion of a sub\-sequence of $\baseen$ &  $\basexn$ is a per\-mu\-ta\-tion of a signed sub\-sequence &$\basexn$ is a sequence of dis\-jointly finitely supported vectors of $\nnormbuno_E$\\
          
        \hline\hline
    \end{tabular}
}

\

\caption{ \small{$\aset$-embeddings for an admissible set determined by an admissible family $\admfam$.}}
\label{tab:table2}
\end{table}
\end{center}

\newpage

We can summarize the notions of $\asetf{E}$-minimality which follows from each non-trivial  $\asetf{E}$-embedding notion given in Table \ref{tab:table2}. For this we first give a few simple definitions. 

In \cite{onaquestion} was defined that a basis $\baseen$ for a Banach space $E$ to be block equivalence minimal  if, and only if, any block sequence have a further block sequence equivalent to $\baseen$.

Recall that two basic sequences $\basexn$ and $\baseyn$ are said to be permutatively equivalent if $(x_n)_n \sim (y_{\sigma(n)})_n$ for some permutation $\sigma$ of the integers. Similarly:

\begin{defi}
Let $\basexn$ and $\baseyn$ be two basic sequences. 
We say that $\basexn$ is signed equivalent to $\baseyn$
if, and only if, there is some $(\epsilon_n)_n \in \{-1,1\}^ \omega$ such that $( x_n)_n \sim (\epsilon_n y_n)_n$. We say that $\basexn$ is signed permutatively equivalent to $\baseyn$ if it is permutatively equivalent to 
$(\epsilon_n y_n)_n$ for 
 some $(\epsilon_n)_n \in \{-1,1\}^ \omega$.
\end{defi}




Recall that a basic sequence is spreading when it is equivalent to all its subsequences. Similarly:

\begin{defi}
    We say that the basic sequence $\baseen$ is signed (resp. permutatively, resp. signed permutatively) spreading if, and only if, $\baseen$ is signed equivalent (resp. permutatively equivalent, resp. signed permutatively equivalent) to all its subsequences. 
\end{defi}


In \cite{onaquestion} it was proved that, as a consequence of the Galvin-Prikry Theorem, if a basis $\baseen$ satisfies that every subsequence has a further subsequence equivalent to $\baseen$, then it is spreading. Adapting the proof of this fact, we shall use the Silver's Theorem to prove the natural form of minimalities in the specific case of subsequences.
All notions of minimalities are summarized in the next Proposition.

\begin{prop}\label{typesofminimality}
    Let $E$ be a Banach space with normalized basis $\baseen$. Then
    \begin{itemize}
    \item Consider the set of blocks $\nnormbasis_E$ for $E$, and $X=[x_n]_n$ where $(x_n)_n$ is a subsequence of $(e_n)_n$. We have
        \begin{itemize}
            \item[$(i)$] $X$ is $bb({\nnormbasis}_E)$-minimal if and only if $\basexn$ is spreading.
            \item[$(ii)$] $X$ is $db({\nnormbasis}_E)$-minimal if and only if $\basexn$ is permutatively spreading.
        \end{itemize}
    \item Consider the set of blocks $\nnormbasis^{\pm}_E$ for $E$, and $X=[x_n]_n$ where is $\basexn$ is a signed subsequence of $\baseen$. We have
        \begin{itemize}
            \item[$(iii)$] $X$ is $bb({\nnormbasis}^\pm_E)$-minimal if and only if $\basexn$ is signed spreading. 
            \item[$(iv)$] $X$ is $db({\nnormbasis}^\pm_E)$-minimal if and only if $\basexn$ is signed permutatively spreading. 
            \end{itemize}
    \item Consider the set of blocks $\nnormbuno_E$ for $E$, and $X=[x_n]_n$ a $\nnormbuno_E$-block subspace of $E$. We have
        \begin{itemize}
            
            \item[$(v)$] $X$ is $bb(\nnormbuno_E)$-minimal if and only if $\basexn$ is block equivalence minimal. 
            \item[$(vi)$] $X$ is $db(\nnormbuno_E)$-minimal if and only if for every $\nnormbuno_X$-block sequence $(y_n)_n$ of $\basexn$ there is a sequence of disjointly supported blocks $\basezn$ of $Y = [y_n]_n$ such that $\basezn \sim \basexn$. 
            \item[$(vii)$] $X$ is $(\nnormbuno_E)^\omega$-minimal if and only if $X$ is minimal.
        \end{itemize}
    \end{itemize}
\end{prop}

\begin{proof}
    First, suppose that  the set of blocks for $E$ is 
    $\nnormbasis_E$, and $X=[x_n]_n$ is a $\nnormbasis_E$-block subspace of $E$, i.e. $\basexn$ is a subsequence of $\baseen$. Then
    \begin{itemize}
        \item[$(i)$] By Definition \ref{defispaceminimal}, $X$ is  $bb_{\nnormbasis}(E)$-minimal  if and only if  for every $(y_{n})_n$ subsequence of $\basexn$, there is a further subsequence $(y_{n_k})_k$ equivalent to $\basexn$, which implies that the sequence $\basexn$ is spreading (see \cite{Ferergodic}). 
    \end{itemize}
    
    Now, consider the set of blocks $\nnormbuno_E$ for $E$, and suppose that $X=[x_n]_n$ is a $\nnormbuno_E$-block subspace of $E$. 
    \begin{itemize}
        \item[$(vii)$] As it was noticed in Subsection \ref{blocksbeingeverything}, we know that  $Y \embast X \iff Y \hookrightarrow X$, when $\asetf{E} = (\nnormbuno_E)^\omega$. So, the conclusion is clear.
        \item[$(v)$] $X$ is $bb(\nnormbuno_E)$-minimal if, and only if, for every $\nnormbuno(X)$-block sequence there is a further $\nnormbuno(X)$-block sequence equivalent to $\basexn$. Therefore, $\basexn$ is a block equivalence minimal basis.
        \item[$(vi)$] simply follows from Definition \ref{defispaceminimal}.
    \end{itemize}
It remains to prove (ii) (iii) (iv),
which are consequence of the facts that if every subsequence of a basis $(e_n)_n$ admits a subsequence which is (ii) permutatively, resp. (iii) signed, resp. (iv) signed permutatively equivalent to $(e_n)_n$, then $(e_n)$ must be (ii) permutatively, resp. (iii) signed, resp. (iv) signed permutatively equivalent to all its subsequences.
We shall only prove (iv), leaving the very similar and easier proofs of (ii) and (iii) as exercises. We follow the proof of the similar lemma from \cite{onaquestion}, with the difference that we shall use that analytic sets in $\eneinf$ are Ramsey. 
    
    Let $\mathcal{C} \subseteq \eneinf$ be such that
    $$\{n_k: k \in \ene\} \in \mathcal{C} \iff (e_{n_k})_k {\rm\ is\ signed\ permutatively\ equivalent\ to\ } (e_k)_k.$$
    
    Let us consider the Polish space $\eneinf \times \{-1,1\}^\omega \times {\rm Bij}(\omega)$, where $\{-1,1\}^\omega$ is equipped with the usual topology, and the Polish topology on the set ${\rm Bij}(\omega)$ of bijections of $\omega$ is induced by its inclusion in $\omega^\omega$. Then $\mathcal{C}$ can be expressed as follows:
    \begin{eqnarray*}
        \mathcal{C} &=& \{\{n_k: k \in \ene\} \in \eneinf : \exists (\delta_k)_k \in \{-1,1\}^\omega, \exists \sigma \in {\rm Bij}(\omega): ((\delta_k e_{n_k})_k \sim (e_{\sigma(k)})_k)\}\\
        &=& \{\{n_k: k \in \ene\} \in \eneinf : \exists (\delta_k)_k \in \{-1,1\}^\omega, \exists \sigma \in {\rm Bij}(\omega): (\{n_k: k \in \ene\},(\delta_k)_k, \sigma) \in \mathcal{B}\}\\
        &=& \text{proj}_{\eneinf}(\mathcal{B}),
    \end{eqnarray*}
   
    where $$\mathcal{B} = \bigcup_{C\geq 1} \bigcap_{k \geq 0} \left\{(\{n_i: i \in \ene\},(\delta_i)_i, \sigma) \in \eneinf \times \{-1,1\}^\omega \times {\rm Bij}(\omega): ((\delta_i e_{n_i})_{i=0}^k \sim_C (e_\sigma(i))_{i=0}^k) \right\}.$$
     
    Since 
    this is a countable union of countable intersections of  open sets in $ \eneinf \times \{-1,1\}^\omega \times {\rm Bij}(\omega)$, the set $\mathcal{B}$ is a Borel subset of $ \eneinf \times \{-1,1\}^\omega \times {\rm Bij}(\omega)$. Therefore, $\mathcal{C}$ is analytic in $\eneinf$. By  Silver's Theorem, there is some $H \in \eneinf$ such that either $[H]^\infty \subseteq \mathcal{C}$ or  $[H]^\infty \subseteq \eneinf\setminus \mathcal{C}$. 
    
    If $[H]^\infty \subseteq \eneinf\setminus \mathcal{C}$ then the sequence $(e_n)_{n \in H}$ satisfies that all its subsequences are not signed permutatively equivalent to $\baseen$, which is a contradiction. On the contrary, if $[H]^\infty \subseteq \mathcal{C}$, then $(e_n)_{n \in H}$ is signed permutatively equivalent to all its subsequences, and so  is $\baseen$ because it is signed permutatively equivalent to $(e_n)_{n \in H}$. 
\end{proof}
Let us end this section with a few obvious implications:

\[
\begin{array}
[c]{ccc}
\baseen \text{ is spreading} &\Longrightarrow & \baseen \text{ is signed spreading}\\
\big\Downarrow & & \big\Downarrow  \\
\baseen \text{ is permutatively spreading}  & \Longrightarrow & \baseen \text{ is signed permutatively spreading}  
\end{array}
\]
also,
\begin{center}
      $\baseen$ is  block equivalence minimal  $\implies$ $\baseen$ is  $db(\nnormbuno_E)$-minimal $\implies$ $E$ is  minimal.
\end{center}

The canonical basis of $c_0$ and $\ell_p$, with $1\leq p < \infty$, is, in each case, block equivalence minimal. In \cite{androschlum} it was proved that the canonical basis of the Schlumprecht space $\mathcal{S}$ is block equivalence minimal. 
In  \cite{Fer} it was observed  that $\mathbf{T}^\ast$ has no  ``block minimal'' block subspaces, and so in particular does not have 
block equivalence minimal block subspaces. 
It may actually be seen that $\mathbf{T}^\ast$ contains no block subspace with (vi). Here is a sketch of the proof:
$\mathbf{T}^\ast$
is ``strongly asymptotically $\ell_\infty$'' (see \cite{CO} and also \cite{DFKO}), which means that $n$ normalized disjointly supported vectors supported far enough on the basis are equivalent to the natural basis of $\ell_\infty^n$;
on the other hand, a standard diagonalization argument (see e.g. Lemma \ref{lemma2.2}) gives
that a block subspace with (vi) must have a further block subspace $(x_n)_n$ with the uniform version of (vi), i.e. in any block-sequence $(y_n)_n$, the existence of disjointly supported blocks $K$-equivalent to $(x_n)$ for some fixed $K$; the conjunction of the two implies that $(x_n)_n$ must be $K$-equivalent to the unit basis of $c_0$, a contradiction with the fact that
$\mathbf{T}^\ast$ does not contain a copy of $c_0$.
In conclusion, $\mathbf{T}^\ast$ satisfies $(vii)$ and does not satisfy the minimality conditions of $(v)$ (nor $(vi)$) in Proposition \ref{typesofminimality}.  We do not know of spaces satisfying condition $(vi)$ of minimality but not being block equivalence minimal (not satisfying condition $(v)$). 


\section{Results on A-tightness}\label{sec:Atightness}

\subsection{Notions of tightness}

\begin{defi}\label{defiYtightenX}
    Let $E$ be a Banach space with normalized basis $\baseen$. Let $\nnormb_E$ be a set of blocks for $E$ and $\asetf{E}$ an admissible set for $E$. Suppose that $X=[x_n]_n$ is a  $\nnormb_E$-block subspace. We say that a Banach space $Y$ with Schauder basis is $\asetf{E}$-tight in the basis  $\basexn$ if, and only if, there is a sequence of successive intervals $(I_i)_i$ such that for every $A \in \eneinf$
    \begin{equation}
        Y \nembast [x_n: n \notin \cup_{i \in A} I_i].
    \end{equation}
\end{defi}
\begin{defi}\label{defiAtight}
Let $E$ be a Banach space with a normalized basis $\baseen$. Let $\nnormb_E$ be a set of blocks for $E$ and $\asetf{E}$ an admissible set for $E$. Suppose that $X=[x_n]_n$ is a  $\nnormb_E$-block subspace. The basis $\basexn$ is $\asetf{E}$-tight if, and only if, every $\nnormb_X$-block subspace $Y$ of $X$ is $\asetf{E}$-tight in the basis $\basexn$. The $\nnormb_E$-block subspace $X$ is $\asetf{E}$-tight if, and only if, $\basexn$ is an $\asetf{E}$-tight basis.
\end{defi}

\begin{obs}
    Let $E$ be a Banach space with a normalized basis $\baseen$. Let $\nnormb_E$ be a set of blocks for $E$ and $\asetf{E}$ an admissible set for $E$. From Definition \ref{defiAtight}, $E$ is $\asetf{E}$-tight if, and only if, every $\nnormb_E$-block subspace $X$ is $\asetf{E}$-tight in $\baseen$.
\end{obs}

The following result extends Proposition 3.1 of Ferenczi - Godefroy \cite{TightBaire} for the original notion of tightness. To prove it we use Corollary 2.4 in \cite{TightBaire} where the following characterization of meager and comeager subsets of the Cantor space is given and it is the following: let $B$ be a subset of $2^\omega$ closed by supersets. 
    \begin{itemize}
        \item[$(i)$] B is meager if, and only if, there exist a sequence $(I_i)_i$ of successive intervals in $\ene$ such that 
        $$\seqq{u} \in B \Rightarrow \{n \in \omega: \supp{}{\seqq{u}} \cap I_n = \emptyset\} \text{ is finite.}$$
        \item[$(ii)$] B is comeager if, and only if, there exist a sequence $(I_i)_i$ of successive intervals in $\ene$ such that 
        $$ \{n \in \omega:  I_n\subseteq \supp{}{\seqq{u}}\} \text{ is infinite} \Rightarrow \seqq{u} \in B $$
    \end{itemize}

\begin{prop}\label{EyAtight}
     Let $E$ be a Banach space with normalized basis $\baseen$. Let $\nnormb_E$ be a set of blocks for $E$ and $\asetf{E}$ an admissible set for $E$. Suppose that $X=[x_n]_n$ is a  $\nnormb_E$-block subspace and $Y$ is a $\nnormb_X$-block subspace of $X$. Then, $Y$ is $\asetf{X}$-tight in $\basexn$  if, and only if, the set 
    \begin{equation}\label{eqEy}
        E_{Y,X}^\aset : = \{\seqq{u} \in 2^\omega: Y \embast [x_n: n \in \supp{}{u}]\}
    \end{equation}
    is meager in $2^\omega$.
\end{prop}
\begin{proof}
   Let $E$, $\nnormb_E$ and $\asetf{E}$ be as in the hypothesis.  Suppose that $X=[x_n]_n$ is a  $\nnormb_E$-block subspace and let $Y$ be a $\nnormb_X$-block subspace of $X$. If $Y$ is $\asetf{E}$-tight in $\basexn$, then there are intervals $I_0 < I_1< ...$ such that for any $A \in \eneinf$, 
        \begin{equation}\label{eqblocksubspace1}
            Y \nembast [x_n : n \notin \cup_{i \in A}I_i].    
        \end{equation}
        Let $\seqq{u} \in E_{Y, X}^\aset$ (clearly $\supp{}{\seqq{u}}\in \eneinf$) and suppose by contradiction that $A_\seqq{u} =\{i \in \ene: I_i \cap \supp{}{\seqq{u}} = \emptyset\}$ is infinite. We have
        $$\supp{}{\seqq{u}} \subseteq  \ene \setminus \bigcup_{i \in A_u}I_i. $$ 
        By the observation after Definition \ref{defiAembed}, we obtain
        $$Y \embast [x_n: n \in \supp{}{\seqq{u}}] \Rightarrow Y \embast [x_n: n \notin \cup_{i \in A_u}I_i], $$
        contradicting   Equation \eqref{eqblocksubspace1}. Therefore $A_u $ is finite and, by Corollary 2.4 in \cite{TightBaire}, $E_{Y}^\aset$ is meager in $2^\omega$. 
    
    For the opposite implication, suppose that $E_{Y}^\aset$ is meager in $2^\omega$. Then, by Corollary 2.4 in \cite{TightBaire}, there are subsets $I_0 < I_1< ...$ such that if $\seqq{u} \in E_{Y}^\aset$, then $ \{i \in \ene: I_i \cap \supp{}{\seqq{u}} = \emptyset\}$ is finite. If there is $A \in \eneinf$ such that $Y \embast [x_n : n \notin \cup_{i \in A}I_i]$, then take $v = \ene \setminus \cup_{i \in A}I_i$. Clearly $\Chi_v \in E_{Y}^\aset$ and $\{i \in \ene: I_i \cap v = \emptyset\}$ is infinite, which contradicts that $E_{Y}^\aset$ is a meager subset of $2^\omega$.
\end{proof}

The following lemma uses the same scheme as in \cite{TightBaire} to prove that the set $E_Y = \{u \subseteq \omega: Y \hookrightarrow [x_n: n \in u]\}$ is meager or comeager.

\begin{lemma}\label{lemamagro2}
    Let $E$ be a Banach space with normalized basis $\baseen$. Let $\nnormb_E$ be a set of blocks for $E$ and $\asetf{E}$ an admissible set for $E$. Suppose that $X=[x_n]_n$ is a  $\nnormb_E$-block subspace and $Y$ is a $\nnormb_X$-block subspace of $X$. Then $E_{Y,X}^\aset$ defined in Equation \eqref{eqEy} is either meager or comeager in $2^\omega$.
\end{lemma}
\begin{proof}
     As it is affirmed in Example 2.2 in \cite{TightBaire}, the relation $E_0' $ defined on $\partes(\omega)$ as follows
    $$u E_0' v \iff \exists n\geq 0 ((u \cap [n, \infty) = v \cap [n, \infty)) \; \& \;( |u \cap [0, n-1]|= |v \cap [0, n-1]|)) $$
    is an equivalence relation and it equivalence classes are the orbits of the group $G_0'$ of permutations of $\ene$ with finite support. Once we see $\partes(\omega)$ as the Cantor space, it is Polish and clearly $G_0'$ satisfies that for any $U$ and $V$ non-empty open subsets of $\partes(\omega)$, there is $g \in G$ such that $g(U) \cap V \neq \emptyset$. 
    We want to use the first topological 0-1 law (Theorem \ref{01topologicallaw}) to conclude that $E_{Y,X}^\aset$ is meager or comeager in $2^\omega$, or more specifically we have to prove:
    \begin{itemize}
        \item[$(i)$]  $E_{Y,X}^\aset$ has the Baire Property.
        \item[$(ii)$] $E_{Y,X}^\aset$ is $G_0'$-invariant. 
    \end{itemize}
    To prove $(i)$ we shall see that $E_{Y,X}^\aset$ is an analytic subset of $2^\omega$ (See Theorem 21.6 in \cite{Kechris}). Notice that we can write the set $E_{Y,X}^\aset$ as the projection on the first coordinate of the set $B:=\cup_{k \in \omega} B_k$, where for each $k \in \omega$
    \begin{eqnarray*}
        B_k := \{(u,\basewn) \in 2^\omega \times \nnomega{X}: (y_n)_n \sim_k \basewn \: \& \:\basewn \in [x_i: i \in u] \: \& \: \basewn \in \asetf{X} \}.
    \end{eqnarray*}
    Each $B_k$ is Borel in $ 2^\omega \times \nnomega{X}$ since the relation of two sequences being equivalent is closed and $\asetf{X}$ is a closed subset of $(\nnormb_X)^\omega$.

    In order to prove $(ii)$ we shall see that $E_{Y,X}^\aset$ is $E_0'$-saturated (this is sufficient because the orbits of the group $G_0'$ coincide with the equivalence classes of the relation $E_0'$), that is:
    \begin{equation*}
        E_{Y,X}^\aset = {E_{Y,X}^\aset}^{E_0'} :=\{v \subseteq \omega: \exists u \in E_{Y,X}^\aset (u E_0' v) \}.
    \end{equation*}
    Clearly, $ E_{Y,X}^\aset \subseteq {E_{Y,X}^\aset}^{E_0'} $. Take $v \in {E_{Y,X}^\aset}^{E_0'}$ and let $u \in E_{Y,X}^\aset$ be such that $u E_0' v$. Notice that there is $M$ such that $u$ and $v$ only differ on $M$ elements and $(x_n)_{n \in u}$ and $(x_n)_{n \in u}$ are $\nnormb_X$-block sequences. So, by Proposition \ref{cdem}, there is $K\geq 1$ such that  $(x_n)_{n \in u} \sim_K (x_n)_{n \in v} $. Let $T$ be such $K$-isomorphism from $X_u:= [x_n]_{n \in u}$ to $X_v:= [x_n]_{n \in v}$. By Proposition \ref{propadmpro}, part $(ii)$ we have $T[\asetf{X_u}] = \asetf{X_v}$. Therefore,
    \begin{eqnarray*}
        u \in E_{Y,X}^\aset &\Rightarrow& \exists (z_n)_n \in \asetf{X_u} (\baseyn \sim \basezn)\\
        &\Rightarrow& \baseyn \sim (T(z_n))_n \text{ and }  (T(z_n))_n \in \asetf{X_v}\\
        &\Rightarrow& v \in  E_{Y,X}^\aset.
    \end{eqnarray*}
\end{proof}


\begin{prop}\label{blocksubAtight2}
    Let $E$ be a Banach space with normalized basis $\baseen$. Let $\nnormb_E$ be a set of blocks for $E$ and $\asetf{E}$ an admissible set for $E$. Suppose that $X=[x_n]_n$ is a  $\nnormb_E$-block subspace, $Y$ is a $\nnormb_X$-block subspace of $X$ and $Z$ is a  $\nnormb_Y$-block subspace. If $Z$ is $\asetf{E}$-tight in $X$, then $Z$ is $\asetf{E}$-tight in $Y$.
\end{prop}

\begin{proof}
     Let $E$, $\nnormb_E$ and $\asetf{E}$ be as in the hypothesis. Suppose that $X = [x_n]_n$ is a  $\nnormb_E$-block subspace, $Y$ is a $\nnormb_X$-block subspace of $X$ and $Z$ is a  $\nnormb_Y$-block subspace. 
    Let us denote by 
    $$E_{Z,X}^\aset := \{u \subseteq \omega :  Z \embast [x_n: n \in u]\}$$
    and
    $$E_{Z,Y}^\aset := \{u \subseteq \omega :  Z \embast [y_n: n \in u]\}.$$
    By hypothesis, we know that $E_{Z,X}^\aset$ is meager in $\partes(\omega)$ after the identification of $\partes(\omega)$ with $2^\omega$. Using  Lemma \ref{lemamagro2}, $E_{Z,Y}^\aset$ is meager or comeager. If it were meager, by Proposition \ref{EyAtight} the demonstration ends. Suppose that $E_{Z,Y}^\aset$ is comeager in $\partes(\omega)$. By Corollary 2.4 in \cite{TightBaire}, there are sequences of successive intervals $(I_i)_i$ and $(J_i)_i$ such that 
    \begin{equation}\label{eqmagro1}
        u \in E_{Z,X}^\aset \Rightarrow \{n \in \omega: u \cap I_n = \emptyset\} \text{ is finite}
    \end{equation}
    and
    \begin{equation}\label{eqcomagro1}
        \{n \in \omega:  J_n\subseteq v\} \text{ is infinite} \Rightarrow v \in E_{Z,Y}^\aset.
    \end{equation}
    Let $A \in \eneinf$ be such that 
    $$\left\{k \in \ene: (\bigcup_{n \in A} \bigcup_{i \in J_n} \supp{X}{y_i}) \cap I_k = \emptyset\right\}$$
    is infinite. Such $A$ exists because each $I_i$ and $J_i$ are finite and each $y_i$ is finitely supported. Let $v =  \bigcup_{n \in A} J_n$, then by  Equation \eqref{eqcomagro1}, we have $v \in E_{Z,Y}^\aset$. If $u = \bigcup_{k \in v} \supp{X}{y_k}$, then 
    $$Z\embast [y_n : n \in v] \Rightarrow Z\embast [x_n : n \in u].$$
    
    The last implication follows from the observation after Definition \ref{defiAembed}. Therefore, $u \in E_{Z,X}^\aset$ but it is disjoint of infinitely many intervals $I_k$'s, contradicting Equation \eqref{eqmagro1}.
\end{proof}

\begin{coro}\label{blocksubAtight}
     Let $E$ be a Banach space with normalized basis $\baseen$. Let $\nnormb_E$ be a set of blocks for $E$ and $\asetf{E}$ an admissible set for $E$. Suppose that $X=[x_n]_n$ is a $\asetf{E}$-tight $\nnormb_E$-block subspace. Then, any $\nnormb_E$-block sequence  $(y_n)_n$ of $(x_n)_n$ is an $\aset$-tight basis. 
\end{coro}

\begin{proof}
Let $Z$ be a $\nnormb_Y$-block subspace of $Y$. Since $Z$ is a $\nnormb_X$-block subspace of $X$ and $Z$ is $\asetf{E}$-tight in $X$, by Proposition \ref{blocksubAtight2}, $Z$ is $\asetf{E}$-tight in $Y$. 
\end{proof}

\begin{teor}\label{AtightnoAminimal}
    Let $E$ be a Banach space with normalized basis $\baseen$. Let $\nnormb_E$ be a set of blocks for $E$ and $\asetf{E}$ an admissible set for $E$. If $X=[x_n]_n$ is a $\asetf{E}$-tight $\nnormb_E$-block subspace, then it contains no $\asetf{E}$-minimal $\nnormb_X$-block subspaces.
\end{teor}

\begin{proof}
    Let $E$, $\nnormb_E$, $\asetf{E}$ and $X$ be as in the hypothesis. By contradiction, suppose $Y=[y_n]_n$ an $\asetf{E}$-minimal $\nnormb_X$-block subspace of $X$. Let $Z= [z_n]_n$ be a $\nnormb_Y$-block subspace of $Y$, so $Z$ is $\asetf{E}$-tight in $X$. By Proposition \ref{blocksubAtight2}, $Z$ is $\asetf{E}$-tight in $Y$, then 
    $$E_{Z,Y}^\aset = \{u \subseteq \omega : Z \embast [y_n: n \in u]\}$$ 
    must be meager in $\partes(\omega)$. 
    
    We shall see that the set $E_{Z,Y}^\aset$ coincides with the subset of all the characteristic functions over infinite subsets of $\ene$, which is comeager what leads us to a contradiction. Let us prove this: suppose $v \subseteq \omega$  infinite, then by the $\asetf{E}$-minimality of $Y$ 
    $$Y \embast [y_n : n \in v],$$
    so, there is $\baseun \in \asetf{E} \cap [y_n : n \in v]$ such that $\baseyn \sim \baseun$. 
    
   We know that  $\baseyn, \baseun, \basezn \in \asetf{Y}$ and by $iii)$ in Proposition \ref{equivc} we have
    $$\basezn \ast_Y \baseyn = \basezn \in \asetf{Y} \Rightarrow \basewn := \basezn \ast_Y \baseun \in \asetf{Y}. $$
    Then, $\basewn$ is a $\nnormb_Y$-block sequence of the basic sequence $\baseun$ (it is not necessarily  a block sequence of $X$ because $\baseun$ is not necessarily a block sequence). Also, each $w_n$ has the same scalars in its expansion than $z_n$. Since $\baseun \sim \baseyn$, we have that $\basezn \sim \basewn$ and also we already know that $\basewn \in \asetf{E} \cap [y_n : n \in v]$. So, $Z \embast [y_n: n \in v]$, which means that $v \in E_{Z,Y}^\aset$. We just proved that $\eneinf$ is contained in $E_{Z,Y}^\aset$, thus they are the same set.
\end{proof}

From Definition \ref{defiYtightenX} we obtain the following observation.
\begin{prop}\label{tighteqAtight}
    Let $E$ be a Banach space with normalized basis $\baseen$. Consider $\nnormbuno_E$ as the set of blocks for $E$ and set $\asetf{E} = (\nnormbuno_E)^\omega $. A $\nnormbuno_E$-block basis $(x_n)_n$ is a $\asetf{X}$-tight if, and only if, $\basexn$ is tight (in the usual sense). 
\end{prop}

In particular:

\begin{coro}[Proposition 3.3 in \cite{FerencziRosendal1}]
If $E$ is a Banach space with normalized tight basis $\baseen$, then $E$ has no minimal subspaces.
\end{coro}

As an exercise, the reader may write the other forms of tightness associated to the set $\nnormb_E$ of blocks 
$\nnormbasis_E$, $\nnormbasis^{\pm}_E$ or $\nnormbuno_E$ respectively, and to the choices $bb(\nnormb_E), db(\nnormb_E)$ and $\nnormb_E^\omega$.
All these forms of tightness will be explicited in the final chapter when we list all dichotomies between minimality and tightness associated to each case.

\section{Games for tightness}\label{sec:relagametight}
In this section the objective is to represent forms of tightness in terms of certain infinite games, as in \cite{FerencziRosendal1}.
Let $\basexn$ and $\baseyn$ be two sequences of successive and finitely supported vectors of $E$. Let $Y= [y_n]_n$ and $X=[x_n]_n$. We write $Y \leq^\ast X$  if there is some $N\geq 1$ such that $y_n \in X $, for every $n \geq N$.  First we need two preliminary lemmas. Lemma \ref{lemma2.2} is a modification of Lemma 2.2 in \cite{FerencziRosendal1} and Lemma \ref{lema3.14} is a modification of Lemma 2.1 in \cite{procpelczar}. In both original cases the result was proved for usual block subspaces. We extend those results to $\nnormb_E$-block subspaces

\begin{prop}\label{lemma2.2}
    Let $E$ be a Banach space and $\nnormb_E$ be a set of blocks for $E$.  Suppose that $X=[x_{n}^{0}]_n$ is a $\nnormb_E$-block subspace and $[x_{n}^{1}]_n \geq [x_{n}^{2}]_n \geq ... $ is a decreasing sequence of $\nnormb_X$-block subspaces. Then, there exists a $\nnormb_X$-block sequence $\baseyn$ such that $\baseyn$ is $\sqrt{K}$-equivalent with a $\nnormb_{X}$-block sequence of $[x_{n}^{K}]_n$, for every $K\geq 1$.
\end{prop}  

\begin{proof}

    Let $X=[x_{n}^{0}]_n \geq [x_{n}^{1}]_n \geq ... $ be a decreasing sequence of $\nnormb_X$-block subspaces as in the hypothesis and $C$ be the basis constant of $(x_{n}^{0})_n$. For $M>0$, consider $c(M,C) $ the constant that exists by Proposition \ref{cdem} applied to $X$. 
    
    For each $K \geq 1$, let $M_K$ be the greatest non-negative integer such that 
    \begin{equation}\label{eq0lemma2.2}
        c(M_K, C) \leq \sqrt{K}.
    \end{equation}
    Using  a diagonal argument, we can find a sequence an increasing sequence of natural numbers $(l_i)_i$ and a $\nnormb_X$-block sequence $(y_n)_n$, with the property that for each $K$ there is some $i\leq M_K$ such that $x_{i-1}^{K}< y_i$ and $(y_m)_{m \geq i}$ is a $\nnormb_X$-block sequence of $[x_{n}^{K}: n\geq i ] $. Therefore, $\baseyn$ differs in $i-1$ terms from the block sequence $(x_{0}^K,x_{1}^K, ..., x_{i-1}^K, y_i, y_{i+1}, ... )$. Therefore, such sequences are $c(M_K, C)$-equivalent and by Equation \eqref{eq0lemma2.2} they are $\sqrt{K}$-equivalent.
\end{proof}

\begin{lemma}\label{lema3.14}
    Let $E$ be a Banach space and $\nnormb_E$ a set of blocks for $E$. Suppose that $X$ is a $\nnormb_E$-block subspace. Let $N$ be a countable set and let $\mu : bb_\nnormb(X) \rightarrow \mathbbm{P}(N)$ satisfying either of the following monotonic conditions:
    $$V \leq^{\ast} W \Rightarrow \mu(V) \subseteq \mu(W)$$
    or 
     $$V \leq^{\ast} W \Rightarrow \mu(V) \supseteq \mu(W).$$
    Then, there exists an ``stabilizing'' $\nnormb_X$-block subspace $V_0 \leq E$, i.e. a $\nnormb_X$-block subspace such that $\mu(V) = \mu(V_0)$, for all $V \leq^{\ast} V_0$.
\end{lemma}

\begin{proof}
    If $\mu$ is monotonic increasing, suppose by contradiction that for every $\nnormb_X$-block subspace $W$, there is $V\leq^\ast W$ such that $\mu(V) \subsetneqq \mu(W)$. It is possible to construct a transfinite sequence $(W_\gamma)_{\gamma< \omega_1}$ of  $\nnormb_X$-block subspaces such that if $ \gamma< \eta<\omega_1$, then $W_\eta \leq^\ast W_\gamma$ and $\mu(W_\eta) \subsetneqq \mu(W_\gamma)$.
    
    The sequence $(\mu(W_\eta))_{\eta<\omega_1}$ obtained is an uncountable strongly decreasing chain (with respect to the inclusion) of subsets of $N$, which contradicts that $N$ is a countable set. If $\mu$ is a  monotonic decreasing function, the result follows analogously. 
\end{proof}
We now define asymptotic games in same vein as in \cite{FerencziRosendal1}, which a careful  choice of the sets of blocks in which the players are allowed to choose their moves.
\begin{defi}
    Let $E$ be a Banach space with normalized basis $\baseen$, $\nnormb_E$ be a set of blocks for $E$ and $\asetf{E}$ be an admissible set for $E$. Let $X= [x_n]_n$ be a $\nnormb_E$-block subspace, and let $Y$ be a  Banach space with normalized basis $\baseyn$. Suppose $C\geq 1$. We define the asymptotic game $\hast$ with constant $C$ between players ${\rm I}$ and ${\rm II}$ taking turns as follows: ${\rm I}$ plays a natural number $n_i$, and ${\rm II}$ plays a natural number $m_i$ and a not necessarily normalized $\nnormb_X$-block vector $u_i \in X[n_0, m_0] +... +X[n_i, m_i]$, where $X[k, m]:=[x_n : k \leq n \leq m]\cap \nnormb_X$, for $k \leq m$ natural numbers. Diagramatically,

    \begin{tabular}{ c c c c c c c c }
        \bf{I} & & $n_0$ &  & $n_1$ &  &  & ... \\ 
        \bf{II} &  &  & $m_0$, $u_0$ &  & $m_1$, $u_1$ &  & ...
    \end{tabular}
    
    The sequence $\baseun$ is the outcome of the game and we say that ${\rm II}$ wins the game $\hast$ with constant $C$, if $(u_n)_n \sim_C \baseyn$ and $\baseun \in \asetf{X}$.
\end{defi}

The game $\hast$ with constant $C$ is determined since it is equivalent to a Gale-Stewart game, which is open for player I; we shall say that the game  $\hast$ with constant $C$ is open for player I.  Notice that if ${\rm II}$ has a winning strategy for the game $\hast$ with constant $C$, then for any sequence $(I_i)_i$ of successive intervals we have $Y \embast_C (X, I_i)$. Therefore, if ${\rm II}$ has a winning strategy for the game $\hast$ with constant $C$ then $Y$ is not $\asetf{E}$-tight in $X$.

The following definition is similar to the one used in \cite{FerencziRosendal1}.

\begin{defi}\label{nota1}
    Let $E$ be a Banach space with normalized basis $\baseen$, $\nnormb_E$ be a set of blocks and $\asetf{E}$ an admissible set for $E$. Let $X=[x_n]_n$ be a $\nnormb_E$-block subspace, $Y$ be a Banach space and $(I_i)_{i}$ be a sequence of successive non-empty intervals of natural numbers.
    \begin{itemize}
        \item[$(i)$] Let $K$ be a positive constant. We write
        $$Y \embast_K (X, I_i)$$
        if there is $A \in \eneinf$ containing 0, such that $Y \embast_K [x_n, n \notin \cup_{i \in A}I_i]$.
        \item[$(ii)$] We write 
        $$Y \embast (X, I_i)$$
        if there is $A \in \eneinf$ such that $Y \embast [x_n, n \notin \cup_{i \in A}I_i]$.
    \end{itemize}
\end{defi}

\begin{obs}
Notice that under the hypothesis of the Definition \ref{nota1}, if there is some $A \in \eneinf$ such that $Y \embast [x_n, n \notin \cup_{i \in A}I_i]$ and $0 \notin A$, then there is some $B \in \eneinf$ containing $0$ such that $Y \embast [x_n, n \notin \cup_{i \in B}I_i]$.
\end{obs}

In the original paper of Ferenczi-Rosendal, special attention is given to the (Borel, continuous,...) dependence of the sequence $I_j$ of intervals associated to a subspace $Y$ in the definition of tightness. This has application for considerations on the classification of the isomorphism relation between subspaces and the so-called ``ergodic space" problem \cite{Ferergodic}, as in \cite{FerencziRosendal1} Theorem 7.3. In the present paper we are not considering these aspects, which allows to give a simplify certain parts of the proof - there is no reference to a Borel or continuous map defining those intervals as in the notion of continuous tightness (\cite{FerencziRosendal1} p165).
On the other hand, although the general scheme of the proof is the same, special attention has to be given to the roles of the set of blocks and of the type of embeddings to generalize the tight- minimal dichotomy from \cite{FerencziRosendal1}. Approximation properties work similarly, but diagonalization properties must be ensured, as well as the topological properties (closed, open) of the outcomes, and this requires a careful definition of the infinite games at hand.

\begin{lemma}\label{lemma3.7}
    Let $E$ be a Banach space with normalized basis $\baseen$ and  $(\nnormb_E, \asetf{E})$ be an admissible system of blocks for $E$. Suppose that $X=[x_n]_n$ is a $\nnormb_E$-block subspace and that $K$ and $\varepsilon$ are positive constants such that for every $\nnormb_X$-block subspace $Y$ of $X$ there is a winning strategy for player $I$ in the game $\hast$ with constant $K+\varepsilon$. 
Then, for every $\nnormb_X$-block subspace $Y$ there exist a sequence of successive intervals $(I_j)_j$ such that $Y\nembast_K (X,I_j)$.
\end{lemma}

\begin{proof}
   Suppose $X=[x_n]_n$ is a $\nnormb_E$-block subspace, $K$ and $\varepsilon$ as in the hypothesis. We will divide this proof in six steps:
    \begin{itemize}
    
    \item[1.] By hypothesis, for each $\nnormb_X$-block subspace $Y$ of $X$ there is a winning strategy $\sigma_Y$ for player $I$ in the game $\hast$ with constant $K+\varepsilon$. 
    
    \item[2.] Let $C\geq 1$ be the basis constant of $\basexn$. Let $ \rho = 1+\frac{\varepsilon}{K}$. Now, let $0<\theta<1 $ be such that $(1+\theta)(1-\theta)^{-1} = \rho$. Take $\Delta = (\delta_n)_n$ a sequence of positive numbers such that  $  2 C K^2 \sum_{n \in \ene} \delta_n  = \theta$.

    Let $\basewn$ be a $KC$-basic sequence of not necessarily normalized blocks with $\frac{1}{K}\leq \normm{w_i} \leq K$, for any $i \in \ene$. If $\baseun$ is such that $\forall i \in \ene\; (\normm{w_i - u_i} < \delta_i)$, then
    $$2KC \sum_{n \in \ene}\frac{\normm{w_n - u_n}}{\normm{w_n}}= 2CK^2 \sum_{n \in \ene}\delta_n = \theta <1.$$
    
    Thus, $\baseun \sim_{\rho} \basewn$.

    \item[3.] We shall obtain some collection of sets of vectors $\{D_n: n \in \ene \}$ which will be used in step 4 to assist in the construction of a strategy for player I. Since  $(\nnormb_E, \asetf{E})$ is an admissible system of blocks for $E$,  we have that for $X$, the sequence $(\delta_n)_n$ and $K$, there is a collection $(D_n)_n$ of non-empty sets of vectors of $\nnormb_X$ such that 
    \begin{itemize}
        \item[C-1] For each $n$ and for each $d \in \enefin $ such that there is $w \in \asetf{X}$ with $\supp{X}{w} = d$, we have that there are a finite number of vectors $u \in D_n$ such that $\supp{X}{u} = d$.
        \item[C-2] For every sequence $\wi \in \asetf{X}$ satisfying $1/K \leq \min_{i} \normm{w_i} \leq \sup_{i} \normm{w_i} \leq K$, for all $n$ there is $u_n \in D_n$, such that
            \begin{itemize}
                \item[C-2.1] $\supp{X}{u_n} \subseteq \supp{X}{w_n}$.
                \item[C-2.2] $\normm{w_n - u_n}< \delta_n$.
                \item[C-2.3] $(u_i)_i \in \asetf{X}$. 
            \end{itemize}
    \end{itemize}

    \item[4.] Suppose now that $Y$ is a $\nnormb_X$-block subspace with normalized $\nnormb_X$-block basis $\baseyn$. Suppose that $p = (n_0, u_0, m_0, ..., n_i, u_i, m_i)$, with $u_j \in D_j $ for $j \leq i$ is a legal position in the game $\hast$ in which $I$ has played according to $\sigma_Y$.
    
     \begin{tabular}{ c c c c c c c c c c }
        \bf{I} & & $n_0$ &  & $n_1$ &  &  & ... & $n_i$ &\\ 
        \bf{II} &  &  &  $u_0$, $m_0$ &  & $u_1$, $m_1$ &  & ...& & $m_i$, $u_i$
    \end{tabular}
    
    We write $p< k $ if $n_j, u_j, m_j <k$ for all $j \leq i$. Since $II$ is playing in $\prod_{j \leq i}D_j$, using the condition C-1, for every $k$ there is only a finite number of such legal positions $p$ which satisfies $p < k$. So, for every $k \in \ene$ the following maximum exists:
    \begin{equation}
        \alpha(k):= \max \{k, \max \{\sigma_Y(p): p< k\}\}.
    \end{equation}
    We set $I_k= [k, \alpha (k)]$. 
    The intervals in $(I_k)_k$ are not necessarily disjoint, but it is possible to extract a subsequence of successive intervals of it, with $I_0$ as first element. 
    
    \item[5.] To prove that $Y\nembast_K (X,I_j)$ we shall show that for every $A\in \eneinf$, containing 0, $Y\nembast_K [x_n : n \notin \cup_{k \in A}I_k]$.
    
    By contradiction, suppose  there is $A \in \eneinf$ containing 0 and a sequence of blocks $\basewn \in \asetf{X}\cap [x_n : n \notin \cup_{k \in A}I_k]$ such that 
    
    \begin{equation}\label{eqynwn}
        \baseyn \sim_K \basewn.
    \end{equation}
    Recall that, since  $\baseyn$ is normalized, $\frac{1}{K} \leq \normm{w_n}\leq K$, for all $n \in \ene$. 

    By the step 3, condition C-2, we can find for each $i$ a block $u_i \in  D_i$ such that 
    $\normm{{w}_i-u_i}< \delta_i$, $\supp{X}{u_i} \subseteq \supp{X}{w_i}$, $\baseun \in \asetf{X}$, and
    \begin{equation}\label{equnwn'}
        \baseun \sim_{\rho} ({w}_n)_n.
    \end{equation}

    By Equation \eqref{equnwn'}, $\baseun \sim_{K  \rho } \baseyn$. Considering that $\rho  = 1+\frac{\varepsilon}{K}$, we can conclude that  $\baseun \sim_{K + \varepsilon } \baseyn$.
    \item[6.] Finally, we will construct a playing $\overrightarrow{p}$ in the game $\hast$ with constant $K+\varepsilon$, where player $I$ will follow his winning strategy and the outcome will be the sequence $\baseun$. Which means that  $I$ wins the game, leading us to a contradiction. In order to do that define $n_i$, $m_i$ natural numbers and $a_i \in A$ as follows:
    
    Let $a_0 = 0 $ and $n_0 = \sigma_Y(\emptyset) = \alpha(0)$, then, by definition of $I_k$, $I_0 = [0, \alpha(0)] = [0, n_0]$. Find $a_1 \in A$, such that $n_0, u_0, a_0 <a_1$ and set $m_0 = a_1 - 1$. Then $p_0 = (n_0, m_0, u_0)$ is a legal position in $\hast$ in which $I$ has played according to his winning strategy $\sigma_Y$. Since $w_0 \in X[n_0, m_0]$ and $\supp{X}{u_0} \subseteq \supp{X}{w_0}$, we have $u_0 \in X[n_0, m_0]$. 
    
    Now, as $p_0 < a_1$, by the definition of the function $\alpha$, if $n_1= \sigma_Y(n_0, m_0, u_0)$, we obtain $n_1 \leq \alpha(a_1) $. Therefore, $]m_0, n_1[ = [m_0+1, n_1 -1] = [a_1, n_1-1]\subseteq [a_1, \alpha(a_1)]= I_{a_1}$.
    
    Suppose by induction that $n_0, ..., n_i$, $m_0, ..., m_i$ and $a_0, ..., a_i \in A$ have been defined. Since $[0, n_0[ \subseteq I_0$ and $]m_j, n_j+1[ \subseteq I_{a_{j+1}}$, for all $j<i$, we have 
    $$u_i \in X[n_0, m_0] + X[n_1, m_1] + ... + X[n_i, \infty[.$$
    Find some $a_{i+1} \in A$ greater than  $n_0, ..., n_i$, $u_0, ..., u_i$ and $a_0, ..., a_i$ and let $m_i=a_{i+1}-1$, then
     $$u_i \in X[n_0, m_0] + X[n_1, m_1] + ... + X[n_i, m_i].$$
    Therefore $p_i = (n_0, m_0, u_0, ..., n_i, m_i, u_i)$ is a legal position of the game $\hast$ with constant $K + \varepsilon$ in which $I$ has played according to $\sigma_Y$. Since $p_i < a_{i+1}$, we have 
    $$n_{i+1} = \sigma_Y(n_0, m_0, u_0, ..., n_i, m_i, u_i) \leq \alpha(a_{i+1})$$
    and
    $$]m_i, n_{i+1}[ = [m_i+1, n_{i+1} -1] = [a_{i+1}, n_{i+1} -1]\subseteq [a_{i+1}, \alpha(a_{i+1})]= I_{a_{i+1}}.$$
    Set $\overrightarrow{p}$ the legal run such that each $p_i$ is a legal position for the game. Such $\overrightarrow{p}$ is the run we were looking for to produce a contradiction.
    
    \end{itemize}
    
\end{proof}

The following technical lemma gives us a criterion for passing from the existence of  intervals dependent on $K$ for which $Y$ is not $\aset$-embedded in $(I^{(K)}_j)$ with constant $K$, to the existence of intervals $(J_j)_j$ for which $Y$ is not embedded for any constant
$K$. It is similar to Lemma 3.8 from \cite{FerencziRosendal1}.

\begin{lemma}\label{lema3.8}
     Let $E$ be a Banach space with normalized basis $\baseen$ and  $(\nnormb_E, \asetf{E})$ be an admissible system of blocks for $E$. Suppose that $X=[x_n]_n$ is a $\nnormb_E$-block subspace and $Y$ is a Banach space with normalized basis $\baseyn$. If for every constant $K$ there are successive intervals of natural numbers $(I^{(K)}_n)$ such that $Y\nembast_K (X, I^{(K)}_j)$, then there is a sequence of successive intervals $(J_j)_j$ such that $Y\nembast (X, J_j)$.
\end{lemma}

\begin{proof}
    Let $E$, the pair $(\nnormb_E, \asetf{E})$, $X = [x_n]_n$ and $Y=[y_n]_n$ be as in the hypothesis. 
    We will construct the intervals $(J_j)_j$ inductively. The idea is to find such a sequence satisfying:
    \begin{itemize}
        \item[$(i)$] For each $n \geq 0$, $J_n$ contains one interval of each $(I_{i}^{(n)})_i$.
        \item[$(ii)$] For each $n \geq 1$, if $M = \min J_n -1$ and $K = \lceil{n \cdot c(M)}\rceil$ (where $c(M)$ is the constant which existence is guaranteed by Proposition \ref{cdem} for $(x_n)_n$), then $\max J_n > \max I_{0}^{(K)} +M$.
    \end{itemize}
    This can be done as follows: take $J_0 = I_{0}^{(1)}$.
    Now suppose that we have defined $J_0, ..., J_n$ satisfying $(i)$ and $(ii)$. Let $a$ be a natural number greater than $\max J_n$, put $M=a-1$ and $K = \lceil{(n+1)\cdot c(M)\rceil}$. Take $b >\max I_{0}^{(K)} +M$ and such that there exists $j_i \in \ene$ with $I_{j(i)}^{(i)} \subseteq [a,b]$, for all $i \in \{1, ..., n+1\}$ (this can be done because the intervals are finite and we are looking at just the first $n+1$ sequences). Let $J_{n+1} := [a,b]$. By construction, such $J_{n+1}$ satisfies the conditions $(i)$ and $(ii)$.
    
    By contradiction, suppose that $A \in \eneinf$ and that for some integer $N$, we have
    $$Y \embast_N [x_n: n \notin \cup_{i \in A}J_i].$$
    This implies that there is a sequence $\basewn$ of $\nnormb_X$-blocks in $ \asetf{X}\cap [x_n: n \notin \cup_{i \in A}J_i]$ such that $\baseyn \sim_N \basewn$.
    Pick $a \in A$ such that $a \geq N$ and set $M=\min J_a -1$ and $K = \lceil{a\cdot c(M)}\rceil$. Let us define an isomorphic embedding $T$ from 
    $$[x_n: n \notin \cup_{i \in A}J_i]$$
    into 
    $$ [x_n: \max I_{0}^{(K)}< n \leq \max J_a]  + [x_n: n \notin \cup_{i \in A}J_i \; \& \; n > \max J_a] $$
    by setting 
    \begin{equation}
        T(x_n) = \begin{cases} x_n, & \mbox{if } n> \max J_a \\ 
        x_{\max I_{0}^{(K)}+n+1 }, & \mbox{if } n\leq M. \end{cases}
    \end{equation}
    
    Notice that $T$ is an isomorphism between those two $\nnormb_X$-block subspaces. So, by ii) in Proposition \ref{propadmpro}, we have $(T(w_n))_n \in \asetf{X}$.

    Since $T$ only changes at most $M$ vectors from $\basexn$, it is a $C(M)$-embedding, then,  
    $$\baseyn \sim_N \basewn \sim_{C(M)} (T(w_n))_n,$$ 
    
    and because $N \cdot c(M) \leq a \cdot c(M)\leq K$, we obtain
    \begin{equation}\label{eqlema3.8.1}
        Y \embast_{K} [x_n: n \notin \cup_{i \in A}J_i \; \& \; n > \max J_a] + [x_n: \max I_{0}^{(K)}< n \leq \max J_a]    
    \end{equation}
    Now, since for each $n \geq 1$, $J_n$ contains one interval of each $(I_{i}^{(n)})_i$, for any $l \in A$ such that $l\geq K$ there is $b(l) \in \ene$ such that $I_{b(l)}^{(K)} \subseteq J_l$. Let $B= \{0\} \cup \{b(l): \:l \in A, \: l\geq K\}$. Then, 
    $$id: [x_n: n \notin \cup_{i \in A}J_i \; \& \; n > \max J_a] + [x_n: \max I_{0}^{(K)}< n \leq \max J_a]   \longrightarrow [x_n : n \notin \cup_{i \in B}I_{i}^{(K)}]$$ 
    is an isomorphism onto its image and by ii) in Proposition \ref{propadmpro} and Equation \eqref{eqlema3.8.1} we have:
    \begin{equation*}
         Y \embast_{K} [x_n : n \notin \cup_{i \in B}I_{i}^{(K)}],
    \end{equation*}
    which contradicts our initial hypothesis.
\end{proof}

The next lemma uses a ``diagonalization'' argument to relate the fact that a space $E$ is saturated with $\nnormb_E$-block subspaces $X$ such that for every $Y\leq X$, $I$ has a winning strategy for the game $\hast$ for any constant $K$, with the existence of a $\asetf{E}$-tight $\nnormb_E$-block subspace $X$. It is similar to \cite{FerencziRosendal1} Lemma 3.9, without the study of the Borel dependence of the intervals in the definition of tightness, and in the other hand, with attention to the types of blocks in the construction so that the diagonalization property still holds.



\begin{lemma}\label{lema3.9}
    Let $E$ be a Banach space with normalized basis $\baseen$ and  $(\nnormb_E, \asetf{E})$ be an admissible system of blocks for $E$. Suppose that for every $\nnormb_E$-block subspace $Z$ and constant $K$ there is a $\nnormb_Z$-block subspace $X$ such that for every $\nnormb_X$-block subspace $Y$, $I$ has a winning strategy for the game $\hast$ with constant $K$. Then, there is a $\nnormb_E$-block subspace $X$ which is $\asetf{E}$-tight.
\end{lemma}

\begin{proof}
    Let $E$ and $(\nnormb_E, \asetf{E})$ be as in the hypothesis. The idea of the proof is to construct inductively a sequence $X_0 \geq X_1 \geq X_2\geq ...$ of $\nnormb_E$-block subspaces and corresponding sequences $(I_j^K)_j$ of successive intervals
    such that for all $V \leq X_K$,
    $V \nembast_{K^2} (X_K,I_j^K)$. Once constructed, we will use Proposition \ref{lemma2.2} to obtain the desired $\nnormb_E$-block subspace.
    
    Consider $X_0 = E$  and let $\varepsilon > 0$. 
Assuming defined  $X_0 \geq X_1 \geq ...\geq X_n$,
    and applying the hypothesis to $X_n$, there is a $\nnormb_{X_n}$-block subspace $X_{n+1} \leq X_n$ such that for every $\nnormb$-block subspace $Y\leq X_{n+1}$ and for all $\varepsilon > 0$, $I$ has a winning strategy for the game $H_{Y, X_{n+1}}^{\aset}$ with constant $(n+1)^2+\varepsilon$. By Lemma \ref{lemma3.7}, 
    for every $\nnormb_{X_{n+1}}$-block subspace $V \leq X_{n+1}$, there are intervals $I_j$ 
    for which $V \nembast_{(n+1)^2} (X_{n+1},I_j)$. 
    
    
    Applying  Lemma \ref{lemma2.2} to the sequence
    $$X_0 \geq ...\geq X_K \geq...,$$ 
     we find a $\nnormb_E$-block subspace $X_{\infty} = [x_{n}^{\infty}]_n\leq X_0 = E$, such that for each $K \geq 1$ there is a $\nnormb_{X_K}$-block sequence $(z_{n}^{K})_n$  with $Z_K = [z_{n}^{K}]_n \leq X_K$ such that   
    \begin{equation}\label{eq2lemma3.9}
        (x_{n}^{\infty})_n \sim_{\sqrt{K}} (z_{n}^{K})_n 
    \end{equation}
    
    Let $Y= [y_n]_n \leq X_{\infty}$ be a $\nnormb_E$-block subspace of $ X_{\infty}$. For each $K\geq 1$ there exists a $\nnormb_{Z_K}$-block subspace $V_K = [v_{n}^{K}]_n$ (using the form of the isomorphism given in Equation \eqref{eq2lemma3.9} and $(ii)$ in Proposition \ref{subsetofblocks}) such that 
    \begin{equation}\label{eq3lemma3.9}
       \baseyn \sim_{ \sqrt{K}}  (v_{n}^{K})_n,
    \end{equation}
    and for such $V_K$ we may by construction find $(I_{j}^K)_j$ such that: 
    \begin{equation}\label{eqaux6}
        V_K \nembast_{K^2} (X_{K},I_{j}^K).     
    \end{equation}
    
    \underline{Claim:}  There are successive intervals $(J_{j}^K)_j$ such that 
    \begin{equation}\label{eq1lemma3.9}
        V_K \nembast_{K^2} (Z_{K},J_{j}^K).
    \end{equation}
    
    \begin{proof} (of the claim)
        Let $(n_j)_j$ and $(m_j)_j$ be increasing sequences from $\ene$, such that for each $j \in \ene$ we have
    \begin{itemize}
        \item $n_j < m_j < n_{j+1}$,
        \item there is $k_j>0$ with
    $$\supp{X_K}{z_{n_j}^K} < I_{k_j}^K < \supp{X_K}{z_{m_j}^K}.$$
    \end{itemize}
   
    Let $J_{j}^K = [n_j, m_j]$, for each $j \in \ene$. Such sequences $(n_j)_j$ and $(m_j)_j$ exist because each $I_{j}^K$ and $\supp{X_K}{z_{j}^K}$ are finite subsets. Notice that for each $A \in \eneinf$ we have 
    \begin{equation}\label{eqaux71}
        [z_{n}^K : n \notin \bigcup_{j \in A} J_{j}^K] \subseteq [x_{n}^K: n \notin  \bigcup_{j \in A} I_{k_j}^K].
    \end{equation}
    Now, suppose that there is $B \in \eneinf$ such that 
    $$V_K \embast_{K^2} [z_{n}^K : n \notin \bigcup_{i \in B} J_{i}^K].$$
    
    Then, there is $\basewn \in \asetf{E} \cap [z_{n}^K : n \notin \bigcup_{i \in B} J_{i}^K] $ such that $(v_{n}^K)_n \sim_{K^2} \basewn$. By Equation \eqref{eqaux71}, $\basewn \in \asetf{E} \cap [x_{n}^K: n \notin  \bigcup_{j \in B} I_{k_j}^K] $, then 
    $$V_K  \embast_{K^2} [x_{n}^K: n \notin \cup_{j \in A}  I_{j}^K],$$
    where $A=\{k_j: j\in B\}$, which contradicts Equation \eqref{eqaux6}.
    \end{proof}

    Now, we will show that $Y \nembast_{K} (X_\infty,J_{j}^K)$: suppose, on the contrary, that $Y \embast_{K} (X_\infty,J_{j}^K)$, then, there is $A \in \eneinf$ with $0 \in A$ and a sequence $\basewn \in \asetf{E} \cap [x_{n}^{\infty}: n \notin \cup_{i \in A}J_{j}^K]$ such that 
    \begin{equation}\label{eq.3.9.ynwn}
        \baseyn \sim_K \basewn.
    \end{equation}
    
    Recall that $Y$ and each $Z_K$ are $\nnormb_E$-block subspaces. By the isomorphism given in Equation \eqref{eq2lemma3.9} and using ii) in Proposition \ref{propadmpro}, we can find $(u_{n}^{K})_n \in Z_{K}^\omega$ (image of $\basewn$ by such isomorphism), such that $ (u_{n}^{K})_n \in \asetf{E}\cap [z_{n}^{K}: n \notin \cup_{i \in A}J_{j}^K]$ and 
    \begin{equation}\label{eq.3.9.unwn}
         (u_{n}^{K})_n \sim_{\sqrt{K}} \basewn.
    \end{equation}

    Then, using the Equations \eqref{eq3lemma3.9}, \eqref{eq.3.9.ynwn} and \eqref{eq.3.9.unwn}, we obtain
    
    \begin{equation}
       (v_{n}^{K})_n    \sim_{ \sqrt{K}} \baseyn \sim_K \basewn  \sim_{\sqrt{K}}   (u_{n}^{K})_n.
    \end{equation}
    
    Thus, $(v_{n}^{K})_n \sim_{K^2}  (u_{n}^{K})_n$, which means that 
    $$[v_{n}^{K}]_n = V_K \embast_{K^2} [z_{n}^{K}: n \notin \bigcup_{i \in A}J_{j}^K]. $$
    This contradicts Equation \eqref{eq1lemma3.9}.

    We have proved that, for every $Y \leq X_\infty$ and  for every $K  \geq 1$, there is a sequence of successive intervals $(J_{j}^{K})_j$, such that $Y \nembast_K (X_\infty,J_{j}^{K})$. Using Lemma \ref{lema3.8} there exists a sequence of successive intervals $(L_{i}^{Y})_i$ such that 
    $$Y \nembast (X_\infty,L_{j}^{Y}),$$
    which finishes our proof. 
\end{proof}

\section{Games for minimality}\label{sec:gamesforminimality}

\begin{defi}\label{gast}
     Let $E$ be a Banach space with normalized basis $\baseen$, $\nnormb_E$ be a set of blocks and $\asetf{E}$ an admissible set for $E$. Suppose $L$ and $M$ are $\nnormb_E$-block subspaces of a Banach space $E$ and $C\geq 1$ a constant. We define the asymptotic game $G_{L,M}^\aset$ with constant $C$ between players $I$ and $II$ taking turns as follows. In the $(i+1)$-th round, $I$ chooses a subspace $E_i \subseteq L$, spanned by a finite $\nnormb_L$-block sequence, a not necessarily normalized $\nnormb_L$-block  $u_i \in E_0 + ... + E_i$, and a natural number $m_i$. On the other hand, $II$ plays for the first time an integer $n_0$, and in all successive rounds $II$ plays a subspace $F_i$ spanned by a finite $\nnormb_M$-block sequence, a not necessarily normalized $\nnormb_M$-block vector $v_i \in F_0 + ...+ F_i$ and an integer $n_{i+1}$. 
    
    For a move to be legal we demand that $n_i \leq E_i$, $m_i \leq F_i$ and that for each play in the game, the chosen vectors $u_i$ and $v_i$ satisfy $ (u_0, ..., u_i) \in \asetfin{E}$ and $(v_0, ..., v_i) \in\asetfin{E}$. We present the following diagram:

    \begin{tabular}{ c c c c c c c c }
        \bf{I} & &  & $n_0 \leq E_0 \subseteq L$ &  & $n_1 \leq E_1 \subseteq L$ &  & ... \\ 
         & &  & $u_0 \in E_0$, $m_0$  &  & $u_1 \in E_0 + E_1$, $m_1$ &  &  \\
         & &  &   &  & $(u_0, u_1) \in \asetfin{E}$ &  &  \\
         & &  &  & &  & &  \\
        \bf{II} &  &$n_0$  & &  $m_0 \leq F_0 \subseteq M$  & & $m_1 \leq F_1 \subseteq M$  & ...\\
          &  &  & &  $v_0 \in F_0$, $n_1$ &  & $v_1 \in F_0+F_1$, $n_2$\\
            &  &  & &   &  & $(v_0, v_1) \in \asetfin{E}$
    \end{tabular}

    The sequences $(u_i)_i$ and $(v_i)_i$ are the outcome of the games and we say that $II$ wins the game $G_{L,M}^\aset$ with constant $C$, if $(u_i)_i \sim_C (v_i)_i$.

\end{defi}

In $\gast$ with constant $C$, players $I$ and $II$ must choose $\aset_E$-block subspaces and vectors in $\asetfin{E}$, in contrast to block subspaces and any block vectors as was defined for the game $G_{X,Y}$ with constant $C$ in \cite{FerencziRosendal1}. Also, notice that in the game $G_{L,M}^\aset$ stated in Definition \ref{gast} the outcome  $(u_i)_i$ and $(v_i)_i$ belong to $\asetf{E}$, since for each $n \in \ene$, we have $(u_i)_{i\leq n}, (v_i)_{i\leq n} \in \asetfin{E}$ and $\asetf{E}$ is closed in $(\nnormb_E)^\omega$.

In addition, since the relation of two sequences being equivalent is closed, we know that if $\overrightarrow{p}$ is a legal run in such game satisfying that  every finite stage of $\overrightarrow{p}$ is a finite stage of a run where $II$ wins the game $\gast$ with constant $C$, then $\overrightarrow{p}$ itself is a run where $II$ wins the game $\gast$ with constant $C$. In this sense we say that the winning condition is closed for the player $II$. The next lemma relates the games $\hast$ and $\gast$ with same constant.

\begin{lemma}\label{lemma3.12}
   Let $E$ be a Banach space with normalized basis $\baseen$, $\nnormb_E$ be a set of blocks and $\asetf{E}$ an admissible set for $E$. If $X$ and $Y$ are $\nnormb_E$-block subspaces of $E$ such that player $II$ has a winning strategy for the game $\hast$ with constant $C$, then $II$ has a winning strategy for the game $\gast$ with constant $C$.
\end{lemma}

\begin{proof}
    Let $E$, $\nnormb_E$, $\asetf{E}$ and $C\geq 1$ be as in the hypothesis. Suppose that $X= [x_n]_n$ and $Y=[y_n]_n$ are $\nnormb_E$-block subspaces. We shall exhibit the move of player $II$ after $i$ rounds in the game $\gast$ with constant $C$, and we will prove that such moves determine a winning strategy for $II$ in the game $\gast$ with constant $C$.  
    For each $i$ (even $i=0$), suppose player $I$ has played $i$ times, and we have the following stage in the game $\gast$:
    
    \begin{tabular}{ c c c c c c c c }
        \bf{I} & &   & $0 \leq E_0 \subseteq Y$ & & ...&  & $0 \leq E_i \subseteq Y$     \\ 
         & &  & $u_0 \in E_0$, $m_0$  & &   & & $u_i \in E_0 + ...+ E_i$, $m_i$    \\
          & &  &  & &   & & $(u_0, ..., u_i) \in \asetfin{E}$    \\
         & &  &  & &  & &  \\
        \bf{II} &  &$0$  & &  $m_0 \leq F_0 \subseteq X$  & ... & $m_{i-1} \leq F_{i-1} \subseteq X$  & \\
          &  &  & &  $v_0 \in F_0$, $0$ &  & $v_{i-1} \in F_0 +...+F_{i-1}$, $0$&\\
            &  &  & &   &  & $(v_0, ..., v_{i-1}) \in \asetfin{E}$&\\
    \end{tabular}
    
    Notice that without loss of generality we are asking to player $II$ to play $n_j = 0$ for all $j$ (which may do so since then  player $I$ has more possibilities to play and makes the game more difficult for II). Let us write each block vector $u_j$ as $\sum_{k=0}^{k_j}\lambda_{k}^{j} y_k$, for all $j \leq i$. We can assume that $k_{j-1} < k_{j}$, for all $j \leq i$. 
    
    Consider the following run in the game $\hast$:
    
    \begin{tabular*}{\textwidth}{ p{1mm} p{3mm} p{7mm} p{1mm} p{3mm} p{10mm} p{1mm} p{3mm} p{24mm} p{1mm} p{3mm} p{8mm} p{3mm} p{8mm}p{3mm} p{8mm} p{1mm} p{3mm}}
        \bf{I} & $m_0$ & & ... &$m_0$ & &  ...  & $m_i$ & & ...&  $m_i$   & & $m_i$& & $m_i$& &...&   \\ 
        \bf{II} &   & $p_0, w_0$ & ... & & $p_{k_0}, w_{k_0}$  & ... & &  $p_{k_{i-1}+1}, w_{k_{i-1}+1}$&...  & & $p_{k_{i}}, w_{k_{i}}$ & & $q_0, {w'}_0$ & & $q_1, {w'}_1$ &...\\
    \end{tabular*}

    where $I$ consecutively plays  $m_0$ the first $(k_0 + 1)$-times, then consecutively plays $m_j$ for $(k_j - k_{j-1})$-times, for any $j \in \{1, ..., i\}$, and then he plays $m_i$ constantly. Meanwhile, $II$ moves according to her winning strategy for the game $\hast$ with constant $C$, which, by using the properties of ${\mathcal A}$, guarantees that 
    $$w':= (w_0, ..., w_{k_i}, {w'}_0, {w'}_1, ...) \in \asetf{X}= \asetf{E} \cap X^\omega.$$ 
    
    Since $(u_0, ..., u_i) \in \asetfin{E}$, by condition $d)$ in Definition \ref{admpro}, we have that there is $(t_n)_n \in Y^\omega$ such that $u' = (u_0, ..., u_i, {t}_0, t_1, ...) \in \asetf{Y} = \asetf{E} \cap Y^\omega$.
    Notice that $u' \ast_Y \baseyn = u'\in \asetf{E}$ and $\baseyn \in bb_\nnormb(E) \subseteq \asetf{E}$, thus, using condition $c)$ of Definition \ref{admpro}, we have 
    $$v':= u' \ast_Y w'\in \asetf{E} \cap X^\omega = \asetf{X}. $$
    
    If $v' = ({v'}_j)_j$, then it follows from the inductive construction that:
    \begin{itemize}
        \item ${v'}_j = {v}_j$, for $j <i$,
        \item ${v'}_i = \sum_{k=0}^{k_i} \lambda_{k}^{i} w_k $,
        \item $ ({v'}_0, ..., {v'}_i ) \in \asetfin{X}$.
    \end{itemize}
    
    Set $v_i := {v'}_i $ and 
    $$F_i = X[m_i, \max\{p_{k_{i-1}+1}, ..., p_{k_i}\}].$$
    
    Therefore, $(v_0, ..., v_i ) \in \asetfin{X}$,   $v_i \in F_0 + ... +F_i$, with  $m_i \leq F_i \subseteq X$. This means that $(F_i, v_i, 0)$ is a legal position for $II$ to play in the game $\gast$ with constant $C$ in its $(i+1)$-th round.

    Suppose that we have continued with the game, where $II$ have played by using the previously  procedure in every round, and we have obtained the outcome: $(u_i)_i$ (which $I$ played) and $(v_i)_i$ (which $II$ played). 
    
    Using the closedness condition i) in Proposition \ref{propadmpro}, $\ui$ and $\vi$ are in $\asetf{E}$ (each initial part is in \asetfin{E}). Since $(u_i)_i $ and $(v_i)_i$ are defined with the same coefficients over $(y_i)_i$ and $(w_i)_i$, respectively, we have that $(u_i)_i \sim_C (v_i)_i$.
    Hence, we have showed the moves that $II$ can do in each round to win the game. Consequently, $II$ has a winning strategy for the game $\gast$ with constant $C$. 
\end{proof}

\subsection{An Auxiliar minimal game}

\begin{defi}
    Let $E$ be a Banach space with normalized basis $\baseen$ and $\nnormb_E$ be a set of blocks for $E$. We denote by $\blocksubspaces_E$ the set of subspaces of $E$ generated by a finite $\nnormb_E$-block sequence. 
\end{defi}

\begin{defi}\label{state}
     Let $E$ be a Banach space with normalized basis $\baseen$ and $\nnormb_E$ be a set of blocks for $E$. A state $s$ is a pair $(a,b)$ with $a,b \in (\nnormb_E \times \blocksubspaces_E)^{<\omega}$, such that if $a = (a_0, A_0, ..., a_i, A_i)$ and $b = (b_0, B_0, ..., b_j, B_j)$, then $j = i$ or $j = i-1$. Let us denote by $\mathbf{S}_E$ the (countable) set of states. 
\end{defi}

\begin{obs}
    Let $E$ be a Banach space with normalized basis $\baseen$, $\nnormb_E$ be a set of blocks and $\asetf{E}$ an admissible set for $E$. Take $M$ and $L$ two $\nnormb_E$-block subspaces and $C\geq 1$. Consider the game $G_{L,M}^{\aset}$ with constant $C$. If we forget the integers $m_i'$'s played by $I$ and $n_i$'s played by $II$ in such game, then the set $\mathbf{S}_E$ contains the set of possible positions after a finite number of runs were played.
\end{obs}

\begin{defi}
    Let $E$ be a Banach space with normalized basis $\baseen$, $\nnormb_E$ be a set of blocks and $\asetf{E}$ an admissible set for $E$. Let $M$ and $L$ be two $\nnormb_E$-block subspaces and $C\geq 1$. We say that the state $s = ((a_0, A_0, ..., a_i, A_i),(b_0, B_0, ... b_j, B_j))  \in  \mathbf{S}_E$ is valid for the game  $G_{L,M}^{\aset}$ with constant $C$ if, and only if, the finite sequences $(a_0, ..., a_i), (b_0, ..., b_j) \in \asetfin{E}$.
\end{defi}

\begin{defi}
    Let $E$ be a Banach space with normalized basis $\baseen$, $\nnormb_E$ be a set of blocks and $\asetf{E}$ an admissible set for $E$. Let $M$ and $L$ be two $\nnormb_E$-block subspaces and $C\geq 1$. Consider $s \in \mathbf{S}_E$ a valid state for the game $G_{L,M}^{\aset}$ with constant $C$. We define the game $G_{L,M}^{\aset}(s)$ as the game $G_{L,M}^{\aset}$ with constant $C$ in which the vectors and finite subspaces in the state $s$ have been played in the initial rounds. That is: if $s=(a,b)$ with   $a = (a_0, A_0, ..., a_i, A_i)$ and $b = (b_0, B_0, ..., b_i,B_i)$ then the game $G_{L,M}^{\aset}(s)$ goes like follows:

\begin{tabular}{ c c c c c c  }
    \bf{I} & &  & $n_{i+1} \leq E_{i+1} \subseteq L$ &   & ... \\ 
     & &  & $u_{i+1} \in A_0 + ...+ A_i + E_{i+1}$, $m_{i+1}$  &   &  \\
     & & & $((a_0,..., a_i, u_{i+1}) \in \asetfin{E}$) &   &    \\
     & &  &  & & \\
     \bf{II} &  &$n_{i+1}$  & &  $m_{i+1} \leq F_{i+1} \subseteq M$  &   ...\\
     &  &  & &  $v_{i+1} \in B_0 + ... +B_i + F_{i+1}$, $n_{i+1}$ &  \\
     &  &  & & ($(b_0,..., b_i, v_{i+1}) \in \asetfin{E}$)   &  
\end{tabular}

The outcome of the game will be given by the pair of infinite sequences $(a_0,..., a_i, u_{i+1}, ...)$ and $(b_0,..., b_i, v_{i+1}, ...) $.

If $s=(a,b)$ with   $a = (a_0, A_0, ..., a_i, A_i)$ and $b = (b_0, B_0, ..., b_{i-1},B_{i-1})$ then the game $G_{L,M}^{\aset}(s)$ goes like follows:

\begin{tabular}{ c c c c c c  }
    \bf{I} & & $m_{i}$  &  &  $n_{i} \leq E_{i+1} \subseteq L$ & ... \\ 
     & &  &   &  $u_{i+1} \in A_0 + ... +A_i + E_{i+1}$, $m_{i+1}$ &  \\
     & & &  &  $((a_0,..., a_i, u_{i+1}) \in \asetfin{E}$ ) &    \\
     & &  &  & & \\
     \bf{II} &  & &$m_{i} \leq F_{i} \subseteq M$  &   &   ...\\
     &  &  & $v_{i} \in B_0 + ...+ B_i + F_{i}$, $n_{i}$ &  &  \\
     &  &  & ($(b_0,..., b_{i-1}, v_{i}) \in \asetfin{E}$) &   &  
\end{tabular}

The outcome of the game will be given by the pair of infinite sequences $(a_0,..., a_i, u_{i+1}, ...)$ and $(b_0,..., b_i, v_{i+1}, ...) $. We say that player $II$ wins the game $G_{L,M}^\aset(s)$ with constant $C$ if $(a_0,..., a_i, u_{i+1}, ...) \sim_C (b_0,..., b_i, v_{i+1}, ...)$.    
\end{defi}

\section{Tight-minimal dichotomies}\label{proofofteor}

Now we are ready to prove our main result, Theorem \ref{teor3.13}.
\begin{proof}[Proof of Theorem \ref{teor3.13}]
    
    Let $E$, $\baseen$ and $(\nnormb_E, \asetf{E})$ be as in the hypothesis. We shall prove that if no  $\nnormb_E$-block subspace is   $\asetf{E}$-tight, then there is a $\nnormb_E$-block subspace which is  $\asetf{E}$-minimal.
    
    If $E$ fails to have an $\asetf{E}$-tight subspace then by Lemma \ref{lema3.9} there are a  $\nnormb_E$-block subspace $Z$ of $E$ and a constant $C\geq 1$  such that for every $\nnormb_Z$-block subspace $X$ of $Z$ there is a further $\nnormb_X$-block subspace $Y$ of $X$ such that $I$ has no winning strategy for the game $\hast$ with constant $C$. If we prove that $Z$ has a  $\asetf{E}$-minimal $\nnormb_E$-block subspace the proof will be completed. So, without loss of generality we can suppose that $Z=E$. 
    
    Summing up, we are supposing that for every $\nnormb_E$-block subspace $X$ there is a further $\nnormb_X$-block subspace $Y \leq X$ such that $I$ has no winning strategy for the game $\hast$ with constant $C$. Since the game $\hast$ with constant $C$ is determined, we can conclude that for any $\nnormb_E$-block subspace $X$, there is a $\nnormb_X$-block subspace $Y$ such that $II$ has a winning strategy for the game $\hast$ with constant $C$.

    Let $\tau: bb_\nnormb(E) \rightarrow \mathbbm{P}(\mathbf{S})$ be defined by
    \begin{eqnarray*}
    s \in \tau(M) &\Leftrightarrow& \exists L \; \nnormb_M\text{-block subspace such that player $II$ has a winning strategy for } \\
    & & G_{L,M}^{\aset}(s)  \text{ with constant }C. 
    \end{eqnarray*}
    
    First observe that the elements of $\tau(M)$ are valid states for the game $\gastml$ and that $\tau(M)$ is non-empty for each $M\leq E$ $\nnormb_E$-block subspace: we already saw that there is $L \leq M$ a $\nnormb_M$-block subspace such that $II$ has a winning strategy for the game $\hastml$ with constant $C$, and by Lemma \ref{lemma3.12},  $II$ has a winning strategy for the game $\gastml$ with constant $C$. Then, it is possible to define a valid state $s=(a,b)$, with $b$ being chosen following the winning strategy for II, such that player $II$ has a winning strategy for $\gastml(s)$ and such $s \in \tau(M)$.

    Consider now $M' \leq^\ast M $ a $\nnormb_E$-block subspace and $s \in \tau(M')$. Therefore, there is $L'\leq M'$ a $\nnormb_{M'}$-block subspace such that $II$ has a winning strategy for the game $G_{L', M'}^\aset(s)$ with constant $C$. Since player $I$ can always choose finite subspaces $E_i$'s in $L'$ inside of $M$ and choose integers $n_i$'s large enough to force player $II$ to play in $M'$ and inside of $M$ (the game $G_{L', M'}^\aset$ is asymptotic, in the sense that it does not depend on the first coordinates), it follows that it is possible to find a $\nnormb_M$-block subspace $L\leq M$ such that $II$ has a winning strategy for the game $G_{L, M}^\aset(s)$ with constant $C$. Therefore, $s \in \tau(M)$, and we conclude that $\tau(M') \subseteq \tau(M)$. 
    
    By Lemma \ref{lema3.14} there is a $\nnormb_E$-block subspace $M_0\leq E$ which is stabilizing for $\tau$, i.e. $\tau(M_0) = \tau(M')$, for all $M' \leq^\ast M$  $\nnormb_M$-block subspace.
    
    Define  $\rho: bb_\nnormb(M_0) \rightarrow \mathbbm{P}(\mathbf{S})$ by setting 
    
    $$s \in \rho(L) \Leftrightarrow  \text{ player } II \text{ has a winning strategy for } G_{L,M_0}^{\aset}(s) \text{ with constant }C. $$
    
    Notice that there is a $\nnormb_{M_0}$-block subspace $L \leq M_0$ such that $\rho(L) \neq \emptyset$ (the same justification was used to show that $\tau(M) \neq \emptyset$, for every   $\nnormb_E$-block subspace $M\leq E$), so $\rho$ is a non-trivial function. As we did before, set $L' \leq^\ast L $ a  $\nnormb_{M_0}$-block subspace and $s \in \rho(L)$. If player $II$ has a winning strategy for $G_{L,M_0}^\aset(s)$ then, by the asymptoticity of the game (same previous argument for $\tau$), $II$ has a winning strategy for $G_{L',M_0}^\aset(s)$, so $s \in \rho(L_0')$. Thus $\rho$ is decreasing. We can apply  Lemma \ref{lema3.14} to $\rho$, to find a stabilizing $\nnormb_{M_0}$-block subspace $L_0$ of $M_0$ for $\rho$. Additionally, we obtain that
    \begin{equation}\label{taurho}
        \rho(L_0) = \tau(L_0) = \tau(M_0).
    \end{equation}
    
    Let us prove Equation \eqref{taurho}: since $L_0 \leq M_0$ and $M_0$ stabilizes $\tau$,  $\tau(M_0) = \tau(L_0)$. If $s \in \rho(L_0)$, then player $II$ has a winning strategy for $G_{L_0, M_0}^{\aset}(s)$, which means that $s \in \tau(M_0)$, so $\rho(L_0) \subseteq \tau(M_0)$. If $s \in \tau(M_0) = \tau(L_0)$, then there is some $L'\leq L_0$  $\nnormb_{L_0}$-block subspace such that $II$ has a winning strategy for $G_{L',L_0}^{\aset}(s)$. Since $L_0 \leq M_0$, in particular $II$ has a winning strategy for the game  $G_{L',M_0}^{\aset}(s)$ with constant $C$. Thus, $s \in \rho(L') = \rho(L_0)$ because $L_0$ is  stabilizing for $\rho$.

    \begin{claim}
        For every  $\nnormb_{L_0}$-block subspace $M$, $II$ has a winning strategy for the game $G_{L_0, M}^{\aset}$ with constant $C$.
    \end{claim}
    \begin{proof}[Proof of the claim]
    Fix $M$ a $\nnormb_{L_0}$-block subspace. The idea of the proof of this claim is to show inductively  that for each valid state $s$ from which player $II$ has a winning strategy for the game $G_{L_0, M}^{\aset}(s)$ with constant $C$, there is another state $s'$ which ``extends'' it such that player $II$ has a winning strategy for the game  $G_{L_0, M}^{\aset}(s')$. Then, to use the fact that the winning condition is closed for player $II$  to justify that $II$ has a winning strategy for the game. This method was used  by A. Pelczar in \cite{pelczar} and we are using it in the same way that V. Ferenczi and Ch. Rosendal did  in \cite{FerencziRosendal1}. 
    
    First, let us prove that $(\emptyset, \emptyset) \in \tau(L_0)$. We know that there is a $\nnormb_{L_0}$-block subspace $Y$ such that $II$ has a winning strategy for the game $H_{Y, L_0}^\aset$ with constant $C$. From Lemma \ref{lemma3.12} follows that $II$ has a winning strategy for the game $G_{Y, L_0}^\aset$ with constant $C$, and, by definition of $\tau$, this means that $(\emptyset, \emptyset) \in \tau(L_0)$. Now, we will show that:
    \begin{itemize}
        \item[$(i)$]  For all valid states for the game $G_{L_0, M}^{\aset}(s)$
        $$s=((u_0, E_0, ..., u_i, E_i), (v_0. F_0, ..., v_i, F_i)) \in \tau(L_0),$$
        there is an $n$ (which player $II$ can play), such that for any subspace $E$ spanned by a finite $\nnormb_{L_0}$-block sequence of $L_0$ with support greater than $n$, and any $u \in E_0 + ... + E_i + E$ such that $(u_0, ..., u_i, u) \in \asetfin{E}$ (that is, any move that player $I$ could do in his $(i+1)$-th round in the game $G_{L_0, M}^{\aset}(s)$, disregarding the integer $m_{i+1}$), we have 
        $$((u_0, E_0, ..., u_i, E_i, u, E), (v_0. F_0, ..., v_i, F_i)) \in \tau(L_0).$$
        
        \item[$(ii)$] For any $((u_0, E_0, ..., u_{i+1}, E_{i+1}), (v_0. F_0, ..., v_i, F_i)) \in \tau(L_0)$, and for all $m$, there are $m \leq F$ subspace spanned by a finite $\nnormb_M$-block sequence and $v \in F_0 + ... +F_i + F$ with $(v_0, ..., v_i, v) \in \asetfin{E} $ (which is a legal move that $II$ can play), such that
        $$((u_0, E_0, ...,u_{i+1}, E_{i+1}), (v_0. F_0, ..., v_i, F_i, v, F)) \in \tau(L_0).$$
        This will be the case in which both players has played $(i+1)$-rounds and player $I$ has played in his (i+1)th-move $(E_{i+1}, u_{i+1}, m)$, and it corresponds to player $II$ make a legal move. 
    \end{itemize}

    Let us prove statement $i)$. Suppose that 
    $$s=((u_0, E_0, ..., u_i, E_i), (v_0. F_0, ..., v_i, F_i)) \in \tau(L_0)$$
    By Equation \eqref{taurho}, $II$ has a winning strategy for $G_{L_0, M_0}^\aset(s)$, which means that there is $n$ such that for every  subspace $n \leq E \subseteq L_0$ spanned by a finite $\nnormb_{L_0}$-block sequence and $u \in E_0 + ... + E_i +E$, $II$ has a winning strategy for the game  $G_{L_0, M_0}^\aset(s')$, where
    $$s'=((u_0, E_0, ..., u_i, E_i, u, E), (v_0. F_0, ..., v_i, F_i)).$$
    So, $s' \in \rho(L_0) = \tau(L_0)$.
     
    To prove $ii)$, suppose 
    $$((u_0, E_0, ..., u_{i+1}, E_{i+1}), (v_0. F_0, ..., v_i, F_i)) \in \tau(L_0)$$
    and $m$ is given. Then, as $M\leq L_0 \leq M_0$ and $\tau(M) = \tau(L_0)$, $II$ has a winning strategy for  $G_{L, M}^\aset(s)$, for some $\nnormb_{M}$-block subspace $L\leq M$. Thus, there are $F\leq M$ with $m \leq F$ and $v \in F_0 + ...+F_i +F$ such that $II$ has a winning strategy for  $G_{L, M}^\aset(s')$, where
     $$s'= ((u_0, E_0, ...,u_{i+1}, E_{i+1}), (v_0. F_0, ..., v_i, F_i, v, F))$$
    So, $s' \in \tau(M) = \tau(L_0)$.
    
    Starting at state $(\emptyset, \emptyset) \in \tau(L_0)$ and following inductively those two steps, we can obtain a sequence of states $(s_i)_i$ such that each $s_i \in \tau(L_0)$ is the initial part of the following one $s_{i+1} \in \tau(L_0)$. We can define a strategy for the player $II$ as follows:

    Since $(\emptyset, \emptyset) \in \tau(L_0)$, using $i)$ there is $n_0$ such that whenever $m_0$, $E_0  \leq L_0$ and $u_0 \in E_0$ such that $n_0 \leq E_0$, played by $I$ , we have
    $$((u_0,E_0), \emptyset) \in \tau(L_0).$$
    
    Let $\sigma((\emptyset, \emptyset)) = (n_0)$. Using $ii)$, there is $F_0 \leq M$ and $v_0 \in F_0$ such that 
    
    $$((u_0,E_0),(v_0, F_0)) \in \tau(L_0).$$
    
    Again using $i)$, there is $n_1$ such that whenever $m_1$, $E_1 \leq L_0$ and $u_1 \in E_0+E_1$ such that $n_1 \leq E_1$, played by I, we have
    $$((u_0,E_0, u_1, E_1), (v_0, F_0)) \in \tau(L_0).$$
    Let $\sigma((E_0, u_0, m_0))=(F_0, v_0,n_1)$. Following this process inductively, supposing that player $I$ in the $(k+1)$-th round has played $(E_k, u_k, m_k)$, using $ii)$ there is $F_k\leq M$ and $v_k \in F_0 + ... + F_k$ such that $m_k \leq F_n$ and 
    
    $$((u_0, E_0, ...,u_{k}, E_{k}), (v_0. F_0, ..., v_k, F_k)) \in \tau(L_0).$$
    
    Using $i)$ there is $n_{k+1}$ such that whatever $m_{k+1}$, $E_{k+1}  \leq L_0$ and $u_{k+1} \in E_0 + ... + E_{k+1}$ such that $n_{k+1} \leq E_{k+1}$, played by I, we have
    $$((u_0, E_0, ..., u_{k+1}, E_{k+1}), (v_0. F_0, ..., v_k, F_k)) \in \tau(L_0).$$
    Let $\sigma((E_0, u_0, m_0, ..., E_k, u_k, m_k)) = (F_k, v_k, n_{k+1})$. $\sigma$ is a strategy for $II$ to play in the game  $G_{L_0, M}^{\aset}$ with constant $C$. 
    
    Let $\overrightarrow{p} = (n_0, E_0, u_0, m_0, F_0, v_0, n_1, ...) $ be a legal run of the game  $G_{L_0, M}^\aset$ where $II$ follows the strategy $\sigma$. So, every finite stage $(n_0, E_0, u_0, m_0, F_0, v_0, n_1, ..., E_i,u_i, m_i, F_i, v_i, n_{i+1})$ of $\overrightarrow{p}$ determines the state $s_i = ((u_0, E_0, ...,u_{i}, E_{i}), (v_0. F_0, ..., v_i, F_i)) \in \tau(L_0)  = \rho(L_0)$, which satisfy that player $II$ has a winning strategy for the game $G_{L_0, M_0}^\aset(s_i)$. By construction of $\sigma$, $II$ actually plays in $M\leq L_0 \leq M_0$, so for every $i \in \ene$ $II$ has a winning strategy for the game $G_{L_0, M}^\aset(s_i)$. 
    
    Therefore, for every $i \in \ene$, $p_i$ is a finite stage of a legal run in the game $G_{L_0, M}^\aset$ with constant $C$ where $II$ wins. So, $\overrightarrow{p}$ is a run in the game $G_{L_0, M}^\aset$ with constant $C$ where $II$ wins. Thus, $\sigma$ is a winning strategy for II.
    \end{proof}
    
    Returning to the proof of the theorem: for $L_0$ there is a $\nnormb_{L_0}$-block subspace $Y=[y_n]_n$ such that $II$ has a winning strategy for the game $H_{Y, L_0}^\aset$ with constant $C$. We finish the proof by showing that for every $\nnormb_{L_0}$-block subspace $M\leq L_0$, $Y \embast_{C^2} M$.
    
    Since $II$ has a winning strategy for $H_{Y,L_0}^\aset$ with constant $C$, player $I$ can produce in the game $G_{L_0,M}^\aset$ a sequence $(u_i)_i \in \asetf{L_0}$ such that $(u_i)_i \sim_C (x_i)_i$.
    That is, in each round of the game $G_{L_0, M}^\aset$, player $I$ can choose the pair $(0, u_i)$, where each $u_i$ is obtained by the moves of $II$ in $H_{Y,L_0}^\aset$. By the Claim, $II$ has a winning strategy for the game  $G_{L_0, M}^\aset$ for producing $(v_i)_i \in \asetf{M}$, such that $(u_i)_i \sim_C (v_i)_i$.
    By transitivity $(x_i)_i \sim_{C^2} (v_i)_i$, therefore $Y \embast_{C^2} M$, which ends the proof. 
\end{proof}

\begin{obs} It is interesting to note that our theorem always provides us with a uniform version of ${\mathcal A}$-minimality, namely, there is a constant $K$ such that
$Y$ ${\mathcal A}$-embeds with constant $K$ into any ${\mathcal D}$-block subspace of $Y$. This fact was well-known for usual minimality, i.e. every minimal space must be $K$-minimal for some $K$.
\end{obs}

\subsection{Corollaries from the A-tight - minimal dichotomy}

As a corollary of Theorem \ref{teor3.13} we obtain the third dichotomy of Ferenczi-Rosendal:

\begin{coro}[Third Dichotomy, \cite{FerencziRosendal1}]\label{corothirddicho}
    Let $E$ be a Banach space with normalized basis $\baseen$, then $E$ contains a tight block subspace or a minimal block subspace.
\end{coro}
\begin{proof}
    In Theorem \ref{teor3.13} consider the admissible system of blocks $(\nnormbuno_E, (\nnormbuno_E)^\omega)$. As we already observed in Proposition \ref{tighteqAtight} and in Proposition \ref{typesofminimality} for this particular admissible set we obtain exactly our thesis.
\end{proof}

\begin{coro}\label{cor3.13A}
    Let $E$ be a Banach space with normalized basis $\baseen$, then $E$ contains a block subspace $X = [x_n]_n$ satisfying one of the following properties:
    \begin{itemize}
        \item[(1)] For any  $[y_n]_n\leq X$, there is a sequence $(I_n)_n$ of successive intervals in $\ene$ such that for any $A \in \eneinf$, $[y_n]_n$ does not embed into $[x_n, n \notin \cup_{i \in A}I_i]$ as a block sequence.
        \item[(2)] $(x_n)_n$ is a block equivalence minimal basis.
    \end{itemize}
\end{coro} 
\begin{proof}
   In Theorem \ref{teor3.13} consider the admissible system of blocks $(\nnormbuno_E, bb({\nnormbuno}_E)$ and item $(vi)$ of Proposition \ref{typesofminimality}.
\end{proof}

V. Ferenczi and Ch. Rosendal also remarked in \cite{FerencziRosendal1} that the case of block sequences in this theorem implies the main result of A. Pelczar in \cite{pelczar} and an extension of it due to Ferenczi \cite{Fer}.

\begin{coro}\label{cor3.13B}
    Let $E$ be a Banach space with normalized basis $\baseen$, then $E$ contains a block subspace $X = [x_n]_n$ satisfying one of the following properties:
    \begin{itemize}
        \item[(1)] For any  $[y_n]_n$ block basis of $X$, there is a sequence $(I_n)_n$ of successive intervals in $\ene$ such that for any $A \in \eneinf$, $[y_n]_n$ does not embed into $[x_n, n \notin \cup_{i \in A}I_i]$, as a sequence of disjointly supported vectors.
        \item[(2)] For any $[y_n]_n$ block basis of $X$, $(x_n)_n$ is equivalent to a sequence of disjointly supported vectors of $[y_n]_n$.
    \end{itemize}
\end{coro}
\begin{proof}
    In Theorem \ref{teor3.13} consider the admissible system of blocks $(\nnormbuno_E, ds({\nnormb}_E)$ and item $(vii)$ of Proposition \ref{typesofminimality}.
\end{proof}

Notice that both properties $(1)$ and $(2)$ in Corollary \ref{cor3.13A} and in Corollary \ref{cor3.13B} are incompatible (see Theorem \ref{AtightnoAminimal}).
Corollaries  \ref{cor3.13A} and \ref{cor3.13B} are stated as Theorem 3.16 in \cite{FerencziRosendal1}. In its enunciate is also considered the embedding as a permutation of a block sequence.
%
Nevertheless, as we have already seen in the proofs of this chapter, such kind of embedding corresponds to a non-admissible set (see Proposition \ref{remperm1}). So, the proofs we have presented do not work for the case of the embedding as permutation of a block sequence, and we see no reason to think that last statement is true.

\subsection{Corollaries from the A-tight - minimal dichotomy: subsequences}

We now pass to the case of subsequences, in which we shall see that Ramsey results allow to reduce the number of relevant dichotomies.


\begin{coro}\label{dichosequences3}
    For any basic sequence $\baseen$ in a Banach space, there is $(x_n)_n\preceq \baseen$ satisfying one of the following properties:
    \begin{itemize}
        \item[$(i)$] For any $(y_n)_n \preceq \basexn$ there is a sequence of successive intervals $(I_n)_n$ such that for every $A \in \eneinf$,
        $(y_n)_n$ is permutatively equivalent to no subsequence of $\basexn$ with indices in
        $\ene \setminus \cup_{i \in A}I_i$.
        \item[$(ii)$] $\basexn$ is spreading. 
    \end{itemize}
\end{coro}

\begin{proof}
    In Theorem \ref{teor3.13} consider the admissible system of blocks $(\nnormbasis_E, db_{\nnormbasis}(E))$. The result follows from item $(ii)$ in Proposition \ref{typesofminimality}, with ``permutatively spreading" as a result of (ii). Additionally we use the fact that every permutatively spreading basis admits a spreading subsequence. This fact follows either from the techniques of \cite{Bourgain}, or from a proof similar (and simpler) to the next corollary (Corollary \ref{dichosequencessignos3}).
\end{proof}

       


\begin{coro}\label{dichosequencessignos3}
    For any basic sequence $\baseen$ in a Banach space, there is a subsequence  $(x_n)_n$ of $\baseen$ which satisfies one of the following properties:
    \begin{itemize}
        \item[$(i)$] For any  subsequence  $(y_n)_n $ of $\basexn$ there is a sequence of successive intervals $(I_n)_n$ such that for every $A \in \eneinf$, we have that $\baseyn$ is permutatively signed equivalent to no subsequence of $\basexn$ with indices in 
        $\ene \setminus  \cup_{i \in A}I_i$.

        \item[$(ii)$]  $(x_n)_n$ is signed equivalent to a spreading sequence. 
        
    \end{itemize}
\end{coro}
\begin{proof}
    In Theorem \ref{teor3.13} consider the admissible system of blocks $(\nnormbasis^{\pm}_E, db({\nnormbasis}^\pm_E)$. The result follows from item $(iv)$ in Proposition \ref{typesofminimality}, with ``$\basexn$ signed permutatively equivalent" as the conclusion of (ii). It remains to check that such a $\basexn$ contains a subsequence which is sign equivalent to a spreading sequence. Let $(f_n)_n$ be a spreading model of $\basexn$ (see for example \cite{at}). From the hypothesis we have that for some constant $C$, for any $n$, there is a finite sequence of $n$ signs $(\varepsilon^k_n)$ and a linear order $\leq_n$ on the integers such that $(\varepsilon^1_n x_1,\ldots,
    \varepsilon^n_n x_n) \sim_C (f_1,\ldots,f_n)_{\leq_n}$ - where the notation
    $(f_i)_{\leq}$ means that $span[f_i]$ is equipped with the norm
    $\|\sum_i \lambda_i f_i\|_{\leq}:=\| \sum_i \lambda_i g_i\|$, 
    if $g_1,\ldots,g_n$ is the $\leq$-increasing enumeration of $f_1,\ldots,f_n$.

     By compactness we find an infinite sequence $(\varepsilon_n)_n$ of signs and a linear order $\leq$ on the integers such that
    $(\varepsilon_n x_n) \sim_C (f_n)_{\leq}$. By Ramsey's theorem for sequences of length $2$, we may find an infinite subset $N$
    of the integers such that $\leq$ coincides either with the usual order on $N$,  or with the reverse order.
    In the first case we obtain that $(\varepsilon_n x_n)_{n \in N}$ is $C$-equivalent to the spreading sequence $(f_n)_{n \in N}$ (or equivalently, to
    $(f_n)_n$); in the second case, it is $C$-equivalent to the  basic sequence $(g_n)$ defined by
    $\| \sum_{i=1}^k \lambda_i g_i \|= \| \sum_{i=1}^k \lambda_{k-i} f_i\|$, which is also spreading. This completes the proof.
\end{proof}

This last result is an interesting improvement on combinatorial results involving subsequences. Indeed any basic sequence contains either a signed subsequence which is spreading, or satisfies a very strong form of tightness (involving
changes of signs and permutations)... On the other hand, the following seems to remain unknown.

\begin{question} Let $(x_n)_n$ be a basic sequence such that all subspaces generated by subsequences of $(x_n)_n$ are isomorphic. Must $(x_n)_n$ contain a spreading subsequence?
\end{question}


\bibliographystyle{unsrt}

\end{document}